\documentclass{aims}
\usepackage{amsmath}
  \usepackage{paralist}
\usepackage{subfigure}
\usepackage{graphicx}  \usepackage{epstopdf}
\usepackage{dsfont}
\usepackage{cite}
 \usepackage[colorlinks=true]{hyperref}
\hypersetup{urlcolor=blue, citecolor=red}

\def\currentvolume{X}
 \def\currentissue{X}
  \def\currentyear{200X}
   \def\currentmonth{XX}
    \def\ppages{X--XX}
     \def\DOI{10.3934/xx.xx.xx.xx}

\newtheorem{theorem}{Theorem}[section]
\newtheorem{proposition}{Proposition}
\newtheorem{assumption}{Assumption}
\theoremstyle{definition}
\newtheorem{definition}[theorem]{Definition}
\newtheorem{remark}{Remark}

\title[The IDSA and the Homogeneous Sphere] 
      {The IDSA and the Homogeneous Sphere:\\ Issues and possible
        improvements} 

\author[J\'er\^ome Michaud]{}


 \email{Jerome.Michaud@unige.ch}

\begin{document}

\maketitle

\centerline{\scshape J\'er\^ome Michaud}
\medskip
{\footnotesize
 \centerline{Universit\'e de Gen\`eve, Section de Math\'ematiques,}
   \centerline{2--4, rue du Li\`evre, CP 64,}
   \centerline{ CH-1211 Gen\`eve}
} 



\bigskip

\begin{abstract}
In this paper, we are concerned with the study of the \emph{Isotropic Diffusion Source Approximation} (IDSA) \cite{LiebendoerferEtAl09big} of radiative transfer. After having recalled well-known limits of the radiative transfer equation, we present the IDSA
and adapt it to the case of the homogeneous sphere. We then show that for this example the IDSA suffers from severe numerical difficulties. We argue that these difficulties originate in the min-max switch coupling mechanism used in the IDSA. To overcome this problem we reformulate the IDSA to avoid the problematic coupling. This allows us to access the modeling error of the IDSA for the homogeneous sphere test case. The IDSA is shown to overestimate the streaming component, hence we propose a new version of the IDSA which is numerically shown to be more accurate than the old one.
 Analytical results and numerical tests are provided to support the accuracy of the new proposed approximation.
\end{abstract}

\tableofcontents
\section{Introduction}\label{sec:Intro}
In many astrophysical problems radiation plays a crucial role. As an
example, core-collapse supernov\ae$\ $ (CCSNe) explosion mechanisms
crucially depend on the radiative transfer of neutrinos. In such
practical calculation, the
numerical computation of the full transfer equation is usually too
costly and one has to rely on approximations.

In the context of CCSNe, Liebend\"orfer et
al. \cite{LiebendoerferEtAl09big} have designed the \emph{Isotropic
  Diffusion Source Approximation} (IDSA). This approximation is based
on asymptotic solutions of the radiative transfer equations, see
\cite{BeFrGaLiMiVa12big} for a rigorous derivation of these asymptotic
solutions; the coupling is then performed with the so-called
\emph{Diffusion Source} $\Sigma$ which has the form of a min-max
switch. In \cite{BeFrGaLiMiVa12big}, it is shown that this particular
coupling mechanism might be problematic and mathematical issues are
discussed. The aim of this paper is to study the behavior of the IDSA on a simple radiative
transfer model and try to address some of the
mathematical issues highlighted in \cite{BeFrGaLiMiVa12big}.

The paper is organized as follows. In Section \ref{sec:RadTr} we present the simple radiative transfer model to which we apply the IDSA. We also present the first three angular moments equation associated to it together with the well-known asymptotic limits of these equations, that is the \emph{diffusion
limit}, the \emph{reaction limit} and the \emph{free streaming
limit}.
Section \ref{IDSA} is devoted to the presentation of the IDSA
adapted to our model problem. We discuss the three regimes of this
approximation together with mathematical issues illustrated by
numerical experiments. In Section \ref{sec:HS}, we discuss the
homogeneous sphere test case and its limit of high opacities. We then
reformulate the IDSA in order to remove the diffusion source $\Sigma$
and propose some modifications of the IDSA equations. In Section
\ref{sec:Num}, we perform numerical experiments on the two new
formulations of the IDSA and study their modeling error in the case of
the homogeneous sphere. We conclude in Section \ref{sec:Conclusion} by
discussing the different ways this work can be used in practical
simulations of CCSNe explosions.

\section{Radiative transfer equation and asymptotic limits}\label{sec:RadTr}
In this section, we consider the monochromatic (independent of energy)
radiative transfer equation in spherical symmetry. We assume the
background to be static and exclude inelastic and anisotropic
scattering processes. The radiative transport equation in natural
units ($c = 1$) reads
\begin{equation}\label{eq:RadTransp}
\mathcal L_{\rm f}(f) :=\partial_t f + \mu \partial_rf+\frac{1-\mu^2}{r}\partial_\mu f=
\kappa_{\rm a}(b-f)+\kappa_{s}(J-f),
\end{equation}
where $f(t,r,\mu)$ is the distribution function of the transported
particles that depends on the time $t$, radius $r$ and cosine of the
angle between the propagation direction and the radial direction $\mu$; $\kappa_{\rm a}(r)$ and $\kappa_{\rm s}(r)$ are the absorption
and scattering opacities, respectively; $b(r,\mu)$ is the equilibrium
distribution and $J(t,r)$ is the angular average of $f(t,r,\mu)$
defined in Eq. \eqref{eq:Moments}. We
also introduce the linear transport operator $\mathcal L_{\rm f}$ of
$f$. The first two moment equations are obtained by multiplying
Eq. \eqref{eq:RadTransp} by $1$ and $\mu$ respectively and integrating
over $\mu$. This gives
\begin{equation}\label{EqJ}
\mathcal L_{\rm J}(J,H) :=\partial_t J +\frac{1}{r^2}\partial_r(r^2H)=\kappa_{\rm a}(B-J),
\end{equation}
and
\begin{equation}\label{EqH}
\mathcal L_{\rm H}(J,H,K) :=\partial_t H +\frac{1}{r^2}\partial_r(r^2K) +\frac{K-J}{r}=-(\kappa_{\rm a}+\kappa_{\rm s})H,
\end{equation}
with the standard definitions
\begin{equation}\label{eq:Moments}
\{J, H, K\}:=\frac{1}{2}\int_{-1}^{1} d\mu\ \mu^{\{0,1,2\}}f.
\end{equation}
The two moments equations \eqref{EqJ} and \eqref{EqH} basically
express the conservation of energy and momentum. The three angular
moments $J,H$ and $K$ of the radiation field are proportional to the
radiation energy density, flux and pressure.
We also used the integrated equilibrium distribution $B =
\frac{1}{2}\int_{-1}^{1}b\, d\mu$ and defined the transport ope\-ra\-tors
 $\mathcal L_{\rm J}$ and $\mathcal L_{\rm H}$ for $J$ and $H$,
 respectively. Note that these operators also depends on higher order
 moments and we obtain in general a hierarchy of equations that needs
 to be closed by a prescription of the form
\begin{equation}\label{eq:closure}
H = H(J),\qquad K = K(H,J),
\end{equation}
for a one-moment or a two-moment closure, respectively. A classical
two-moment closure has been introduced in~\cite{LevermorePomraning81}
and other two-moment closures are discussed
in~\cite{SmitHornBludman00} in relation with flux-limited
diffusion. It is also convenient to define the following flux factors.
\begin{definition}[Flux ratio and variable Eddington factor]\label{def:FF}
The \emph{flux ratio} $h(t,r)$ is defined by 
\begin{equation}
h(t,r):=\frac{H(t,r)}{J(t,r)},
\end{equation}
and the \emph{variable Eddington factor} $k(t,r)$ is defined by
\begin{equation}
k(t,r):=\frac{K(t,r)}{J(t,r)}.
\end{equation}
\end{definition}
In order to enable a complete discussion of the IDSA, we need to
discuss the asymptotic behavior of the radiative transfer equation
because the IDSA is based on the coupling of three different
asymptotic limits of this equation as shown in
\cite{BeFrGaLiMiVa12big}. We discuss the diffusion limit, the reaction
limit and the free streaming limit in terms of closure relations
similar to Eq. \eqref{eq:closure}.
\subsection{Diffusion limit}
Following \cite[pp. 350--353]{Mihalas84}, we can use the diffusion ap\-pro\-xi\-ma\-tion of \eqref{eq:RadTransp} by imposing the closure conditions
\begin{equation}\label{eq:diff.closure}
H=-\frac{1}{3\kappa}\frac{\partial J}{\partial r},\qquad
K=\frac{1}{3}J,
\end{equation}
where $\kappa:=\kappa_{\rm a}+\kappa_{\rm s}$ is the total opacity.
These conditions lead to the diffusion equation
\begin{equation}\label{eq:J_sr.diff}
\mathcal L_{\rm J}^{\rm diff}(J):=\partial_t J - \frac{1}{r^2}\partial_r\left(r^2
  \frac{1}{3\kappa}\partial_r J\right) = \kappa_{\rm a}(B-J),
\end{equation}
where we also defined the diffusive transport operator $\mathcal
L_{\rm J}^{\rm diff}$ for $J$. In this limit, the variable Eddington
factor satisfies
\begin{equation}
k^{\rm diff}(t,r) = \frac{1}{3}.
\end{equation}
The diffusion limit is valid in regions where the mean free path
$\kappa^{-1}$ is short compared with the typical length scale of the
system. Such a region will be refered to as \emph{opaque}. This limit can also be obtained by asymptotic expansions, see
for example \cite{BeFrGaLiMiVa12big}. 

\subsection{Reaction limit}
The reaction regime is a special limit of the diffusion limit in which
the total opacity $\kappa$ is so large that the diffusion term can be
dropped in Eq.~\eqref{eq:J_sr.diff}. This corresponds to the following
closure relations
\begin{equation}
H=0,\qquad
K=\frac{1}{3}J.
\end{equation}
The corresponding evolution equation reads 
\begin{equation}\label{eq:J_sr.reac}
\mathcal L_{\rm J}^{\rm reac}(J):=\partial_t J = \kappa_{\rm a}(B-J),
\end{equation}
where we defined the reaction evolution operator $\mathcal L_{\rm
  J}^{\rm reac}$ for $J$ as before.
The reaction limit is valid in regions where the mean free path
$\kappa^{-1}$ is negligible compared with the typical length scale of
the system. As for the diffusion limit, one can use asymptotic
expansion with a different scaling in order to derive this limit, see
for example \cite{BeFrGaLiMiVa12big}.
\subsection{Stationary state free streaming limit}
The stationary state free streaming limit can be obtained by assuming
that the free streaming particles have infinite velocity. As a result,
the change of the distribution function is infinitely fast and one can
drop the time derivative term in Eq. \eqref{EqJ}. We obtain
\begin{equation}\label{eq:J_sr.free}
\mathcal L_{\rm J}^{\rm free}(J):=\frac{1}{r^2}\partial_r(r^2h^{\rm free}J) = \kappa_{\rm a}(B-J),
\end{equation}
where the free streaming flux ratio can be estimated by a geometrical
relation under the assumption that the streaming particles are emitted
from an infinitely opaque homogeneous sphere of radius $R$. In this
case, we have \cite{Bruenn85, LiebendoerferEtAl09big}
\begin{assumption}[Scattering sphere approximation]
The scattering sphere one-moment closure relation is given by the flux ratio
\begin{equation}\label{eq:hfree}
h^{\rm free}(r,R) = \left\{
\begin{array}{ll}
\displaystyle\frac{1}{2},&r<R,\\
\displaystyle\frac{1}{2}\left(1+\sqrt{1-\left(\frac{R}{r}\right)^2}\right),&r\geq R.
\end{array}
\right.
\end{equation} 
The radius $R$ is called the neutrinosphere radius.
\end{assumption}
This flux ratio corresponds, as we will rederive in Section
\ref{sec:HS}, to the flux ratio of an infinitely opaque sphere of
radius $R$. In Section \ref{sec:HS} we also compute the corresponding
variable Eddington factor. The neutrinosphere radius $R$ is usually
computed using the \emph{optical depth} $\tau$
\begin{equation}\label{eq:tau}
\tau(r) = \int_r^\infty \kappa(r) dr
\end{equation}
and the neutrinosphere radius satisfies $\tau(R) = \frac{2}{3}$.

As for the diffusion and the reaction limit, the stationary state
free streaming limit can be derived by asymptotic expansions, see
\cite{BeFrGaLiMiVa12big}. This approximation is valid where the mean
free path $\kappa^{-1}$ is large compared with the typical length
scale of the system. In these regions, the particles propagate freely
and we therefore refer to them as \emph{transparent} regions. In term
of flux factor, the free streaming limit satisfies
\begin{equation}
\lim_{r\to \infty}h^{\rm free}(r,R) = 1, \qquad \lim_{r\to \infty}k^{\rm free}(r,R) = 1.
\end{equation}
These limits come from the fact that in this case, the distribution
function is strongly outward peaked and accumulates on $\mu = 1$.

\section{Isotropic Diffusion Source Approximation (IDSA)}
\label{IDSA}
As mentioned in Section \ref{sec:Intro}, in practical calculations
the direct solve of Eq. \eqref{eq:RadTransp} is too costly and one
need to rely on approximations.
The \emph{Isotropic Diffusion Source Approximation} (IDSA) has been
developed by Liebend\"orfer et al. in \cite{LiebendoerferEtAl09big} in
the context of core-collapse supernov\ae$\ $(CCSNe) modeling. This is
an approximation of the 
radiative transfer of neutrinos in this kind of stellar explosion. It
is based on the observation that in a CCSN event, the inner core of
the star becomes dense enough to be opaque to neutrinos, whereas the
outer layers of the star are transparent. Following the discussion in
Section \ref{sec:RadTr}, the radiative transfer of neutrinos in the
inner core should be well described by 
the diffusion or the reaction limit, whereas in the outer layers it
should be well described by the free streaming limit. This observation
has led to the idea of using an additive decomposition of the
distribution function into two components: a \emph{trapped} component
describing the inner core neutrinos and a \emph{free streaming}
component describing the outer layers neutrinos. Liebend\"orfer and
colleagues also noted that it was enough, for their simulations, to
compute the angular averaged $J$ distribution function. The IDSA is
therefore an heterogeneous ap\-pro\-xi\-ma\-tion of the first angular moment
$J$ of the neutrino distribution function~$f$.
We now recall the hypotheses of the IDSA and state it for our model problem.

\subsection{Ansatz: Decomposition into trapped and streaming neutrinos}
\label{ansatz}

We assume a \emph{decomposition} of
\begin{equation}
\label{decomposition}
f= f^t + f^s
\end{equation}
\emph{on the whole domain} into distribution functions $f^t$ and $f^s$ supposed to account for \emph{trapped} and for \emph{streaming} neutrinos, respectively. This ansatz is also valid for all the moments by linearity of integration
\begin{eqnarray}
J&=&J^t + J^s,\\
H&=& H^t + H^s,\\
K&=& K^t + K^s,\label{decomposition.K}
\end{eqnarray}
where we used the notation introduced in Eq. \eqref{eq:Moments}.

The IDSA is a two components one-moment approximation of
Eq. \eqref{EqJ} that we recall here
\begin{equation}\label{EqJ2}
\mathcal L_{\rm J}(J,H) = \kappa_{\rm a} (B-J).
\end{equation}

\begin{assumption}
Eq. \eqref{EqJ2} can be well approximated by a system of equation of
the form
\begin{eqnarray}
\mathcal L^t_{\rm J}(J^t,H^t(J^t)) &=& \kappa_{\rm a} (B-J^t)
-\Sigma(J^t, J^s, B, \kappa_{\rm a},\kappa_{\rm s})\label{eq:tr}\\
\mathcal L^s_{\rm J}(J^s,H^s(J^s)) &=& -\kappa_{\rm a} J^s
+\Sigma(J^t, J^s, B, \kappa_{\rm a},\kappa_{\rm s})\label{eq:str}\\
\Sigma(J^t, J^s, B, \kappa_{\rm a},\kappa_{\rm s})&:=& \min\left \{
  \max\left[-\frac{1}{r^2}\partial_r\left(\frac{r^2}{3\kappa}\partial_r
      J^t\right)+\kappa_{\rm a}J^s ,0\right],\kappa_{\rm
    a}B\right\},\label{eq:ds} 
\end{eqnarray}
where $\Sigma$ is the \emph{diffusion source}, and $J^t+J^s \approx J$.
\end{assumption}
The IDSA is completed by defining the operators $\mathcal L^t_{\rm J}$
and $\mathcal L^s_{\rm J}$ together with the closure relations for the
trapped and the streaming components.
\begin{assumption}
We assume that $f^t$, $\kappa_{\rm a}$ and $\kappa_{\rm s}$ are
isotropic and that $f^s$ is in the stationary state free streaming
limit. Hence, the IDSA uses the following closure relations
\begin{equation}\label{eq:IDSAclosure}
H^t(J^t) = 0,\qquad H^s(J^s) = g_{\rm idsa} J^s,
\end{equation}
with $g_{\rm idsa} := h^{\rm free}$. Comparing this closure relations
with the asymptotic limits described in Section \ref{sec:RadTr}, we
have the following operator definitions
\begin{equation}\label{eq:IDSAop}
\mathcal L^t_{\rm J}:= \mathcal L_{\rm J}^{\rm reac},\qquad \mathcal
L^s_{\rm J}:= \mathcal L_{\rm J}^{\rm free}.
\end{equation}
\end{assumption}

The IDSA is defined by Eqs \eqref{eq:tr}--\eqref{eq:ds} together with
the closure relations given in Eq. \eqref{eq:IDSAclosure} and the
operator definitions given in Eq. \eqref{eq:IDSAop}. 
\begin{definition}[IDSA]
The IDSA equations are given by (dropping the dependence of $\Sigma$)
 \begin{eqnarray}
\mathcal L^{\rm reac}_{\rm J}(J^t) &=& \kappa_{\rm a} (B-J^t)
-\Sigma\label{eq:tr1}\\
\mathcal L^{\rm free}_{\rm J}(J^s) &=& -\kappa_{\rm a} J^s
+\Sigma\label{eq:str1}\\
\Sigma&:=& \min\left \{
  \max\left[-\frac{1}{r^2}\partial_r\left(\frac{r^2}{3\kappa}\partial_r
      J^t\right)+\kappa_{\rm a}J^s ,0\right],\kappa_{\rm a}B\right\},\label{eq:ds1}
\end{eqnarray}
\end{definition}
The diffusion
source $\Sigma$ is defined in a way to switch between three
different regimes to which we now come.

\subsection{The three regimes of the IDSA}\label{sub:regimes}
The IDSA has three distinct regimes corresponding to the different
values that the diffusion source $\Sigma$ can take:
\begin{enumerate}
\item Reaction regime: 
$$\Sigma = 0;$$
\item Diffusion regime:
$$\Sigma = -\frac{1}{r^2}\partial_r\left(\frac{r^2}{3\kappa}\partial_r
      J^t\right)+\kappa_{\rm a}J^s;$$
\item Free streaming regime: 
$$\Sigma = \kappa_{\rm a}B.$$
\end{enumerate}
We now discuss these three regimes.
\begin{remark}In this work, when speaking about the behavior of the
  IDSA, we will speak about \emph{regimes}. But when speaking about
  the asymptotic behavior of the real distribution, we will speak
  about \emph{limits}. Hence, it is important to distinguish clearly
  between these two terms.
\end{remark}
\subsubsection{Reaction regime}
In the reaction regime, the IDSA equations \eqref{eq:tr1} and
\eqref{eq:str1} become
\begin{eqnarray}
\mathcal L^{\rm reac}_{\rm J}(J_{\rm reac}^t) &=& \kappa_{\rm a} (B-J_{\rm reac}^t)
\label{eq:tr.reac}\\
\mathcal L^{\rm free}_{\rm J}(J_{\rm reac}^s) &=& -\kappa_{\rm a} J_{\rm reac}^s,
\label{eq:str.reac}
\end{eqnarray}
which shows that the trapped component is in the reaction limit, see
Eq. \eqref{eq:J_sr.reac} and $J_{\rm reac}^s$ is solution of
Eq. \eqref{eq:str.reac}, that is
\begin{equation}\label{eq:str.reac.2}
J_{\rm reac}^s(r) = \frac{C}{r^2}{\rm e}^{-r\int_0^r h^{\rm free}(r)\, dr}.
\end{equation}
Because $\frac1{2}\leq h^{\rm free}(r)\leq 1$, $J_{\rm reac}^s(r)> C\frac{{\rm
    e}^{-r^2}}{r^2}$ which diverges in $0$ in an unphysical
way. Therefore, we have $C=0$ and 
\begin{equation}
J^{\rm idsa}_{\rm reac} = J^{t}_{\rm reac}+J^{s}_{\rm reac} =  J^{t}_{\rm reac},
\end{equation}
and the solution of the IDSA in the reaction regime is in the reaction limit.
\subsubsection{Diffusion regime} In the diffusion regime, the IDSA
equations \eqref{eq:tr1} and
\eqref{eq:str1} become
\begin{eqnarray}
\mathcal L^{\rm reac}_{\rm J}(J_{\rm diff}^t) &=& \kappa_{\rm a}
(B-J_{\rm diff}^t-J_{\rm diff}^s)
+\frac{1}{r^2}\partial_r\left(\frac{r^2}{3\kappa}\partial_r
      J_{\rm diff}^t\right)
\label{eq:tr.diff}\\
\mathcal L^{\rm free}_{\rm J}(J_{\rm diff}^s) &=& -\frac{1}{r^2}\partial_r\left(\frac{r^2}{3\kappa}\partial_r
      J_{\rm diff}^t\right),
\label{eq:str.diff}
\end{eqnarray}
or equivalently
\begin{eqnarray}
\mathcal L^{\rm diff}_{\rm J}(J_{\rm diff}^t) &=& \kappa_{\rm a}
(B-J_{\rm diff}^t-J_{\rm diff}^s)
\label{eq:tr.diff1}\\
J_{\rm diff}^s &=& -\frac{1}{3\kappa
    h^{\rm free}}\partial_r
      J_{\rm diff}^t,
\label{eq:str.diff1}
\end{eqnarray}
that is the trapped component is in the diffusion limit up to a term
of absorption of streaming particles. The first equation is obtained
by passing the diffusion term on the left hand side, comparing with
Eq. \eqref{eq:J_sr.diff}. The second is obtained by multiplying
Eq. \eqref{eq:str.diff} by $r^2$ and integrating over $r$. This
modified diffusion limit 
can also be obtained by asymptotic expansions, this has been shown in
\cite{BeFrGaLiMiVa12big}.
\begin{remark}\label{rem:2}
Note that in order to have a positive distribution of streaming
particles $J^s_{\rm diff}\geq 0$, it is necessary to have a decreasing
distribution function of trapped $\partial_r J_{\rm diff}^t \leq
0$. That is, any non-decreasing solution of Eq. \eqref{eq:tr.diff1} is
not physical.
\end{remark}
Substituing \eqref{eq:str.diff1} into \eqref{eq:tr.diff1}, one can
decouple the two equations and the new equation for the trapped
particles is given by
\begin{equation}\label{eq:tr.with.adv}
\partial_t J_{\rm diff}^t - \frac{1}{r^2}\partial_r\left(r^2
  \frac{1}{3\kappa}\partial_r J_{\rm diff}^t\right) -\frac{\kappa_{\rm a}}{3\kappa
    h^{\rm free}}\partial_r J_{\rm diff}^t = \kappa_{\rm a} (B-J_{\rm diff}^t),
\end{equation}
where we explicit the diffusion transport operator $\mathcal L^{\rm diff}_{\rm J}$.

\subsubsection{Free streaming regime} In the free streaming regime,
the IDSA equations \eqref{eq:tr1} and
\eqref{eq:str1} become
\begin{eqnarray}
\mathcal L^{\rm reac}_{\rm J}(J_{\rm free}^t) = \partial_t J_{\rm free}^t &=& 
-\kappa_{\rm a}J_{\rm free}^t,
\label{eq:tr.free}\\
\mathcal L^{\rm free}_{\rm J}(J_{\rm free}^s) &=& \kappa_{\rm a}(B-
J_{\rm free}^s),
\label{eq:str.free}
\end{eqnarray}
which shows that the streaming component is in the stationary free
streaming limit, see Eq. \eqref{eq:J_sr.free}. The trapped equation
can be solve analytically, we obtain
\begin{equation}\label{eq:tr.free2}
J_{\rm free}^t(t) = C{\rm e}^{-\kappa_{\rm a}t},
\end{equation}
That is the trapped component is exponentially decreasing and vanishes
in the stationary state.

\subsection{Mathematical and numerical issues}
\label{sub:issues} 
We have discussed in Subsection~\ref{sub:regimes} that the different
regimes of the IDSA correspond to the asymptotic limit
of the radiative transfer equation. The main problem of the IDSA is
its treatment of transient regions where neither the diffusion, the free
streaming nor the reaction limit apply. Then, nevertheless, one of the
above regimes occurs. In the case of CCSN modeling, a discussion of
this topic by numerical results 
comparing the IDSA with the full radiative transfer equation in the
spherically symmetric case is given in
\cite{LiebendoerferEtAl09big}. In this paper, it is shown that the
IDSA solution approximates pretty well the reference solution computed
by solving the full radiative transfer equation. It is also shown that
the biggest deviations occur in the ``semi-transparent'' transient
region around the neutrinosphere. Similar results have been founded in
a numerical example for a model problem in
\cite{BeFrGaLiMiVa12_esaimbig}. The limiters for the coupling term in \eqref{eq:ds} can be physically
and mathematically motivated, see for example
\cite{LiebendoerferEtAl09big, BeFrGaLiMiVa12big}, but there does not
seem to be a rigorous 
explanation of these limitations by means of the full radiative transfer
equation.

In the following, we take a closer 
look at some properties of the diffusion source $\Sigma$ and show that
this coupling can lead to unphysical solutions and numerical instabilities.

\subsubsection{Spurious trapped}
\begin{figure}[h]
\centering

\subfigure[Solution at time $T = 5$.]{\label{Fig:I-4-SpTr.a}\includegraphics[width=.42\textwidth]{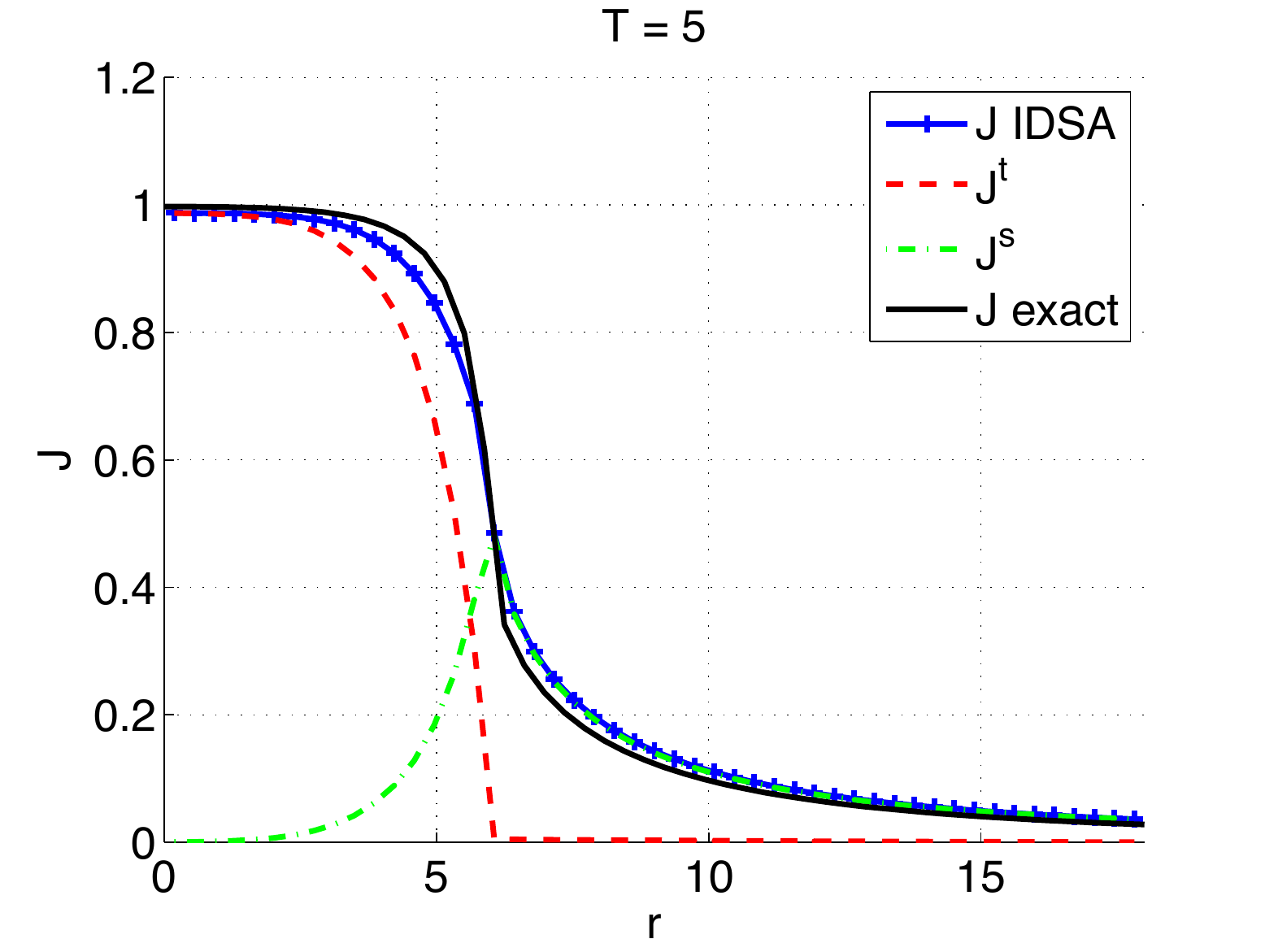}}
\quad
\subfigure[Proportion of trapped $h^t=\frac{J^t}{J^t+J^s}$ and streaming
particles $h^s=\frac{J^s}{J^t+J^s}$ at $T = 5$. ]{\label{Fig:I-4-SpTr.b}\includegraphics[width=.42\textwidth]{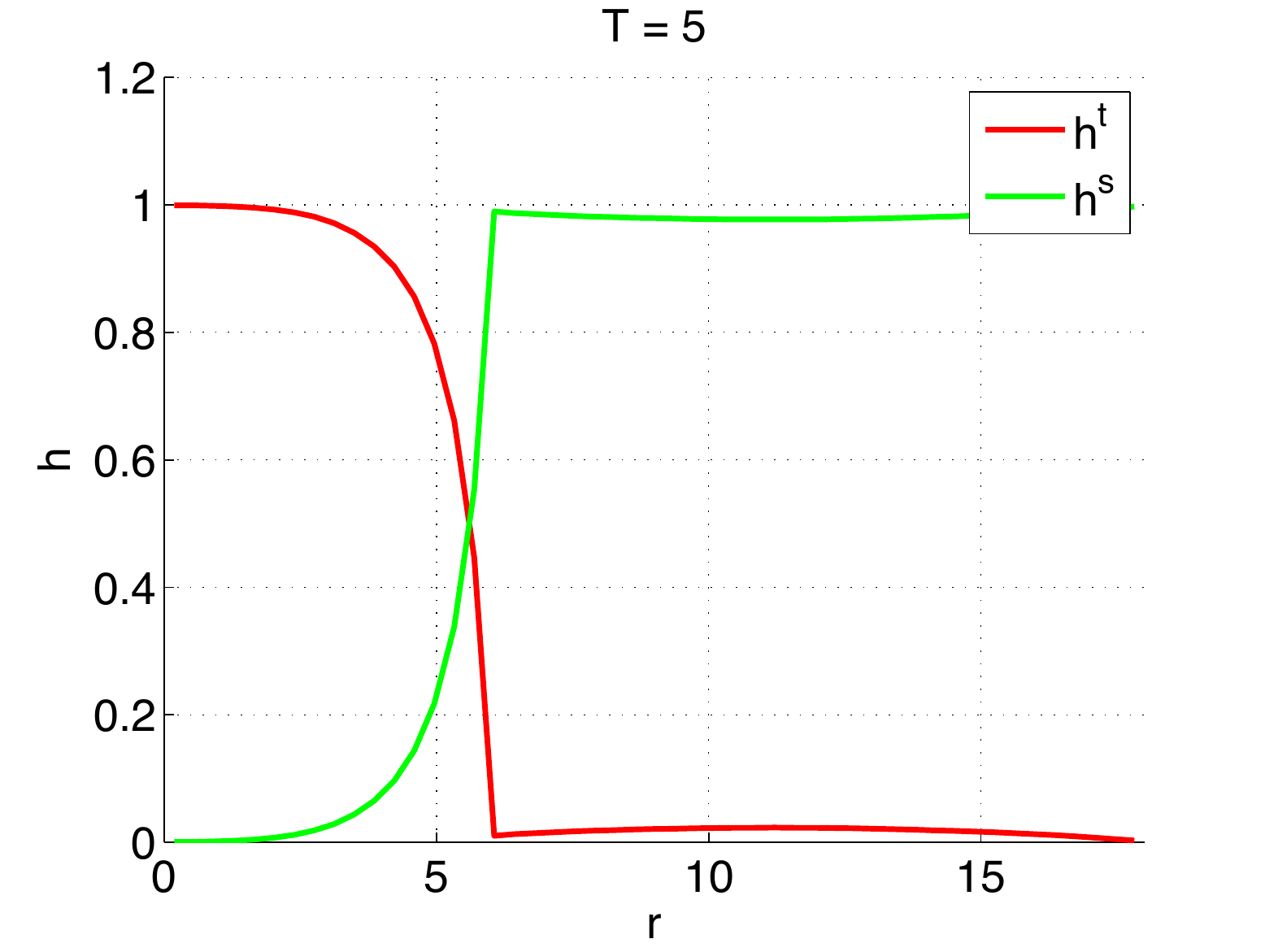}}

\subfigure[Solution at time $T = 500$.]{\label{Fig:I-4-SpTr.c}\includegraphics[width=.42\textwidth]{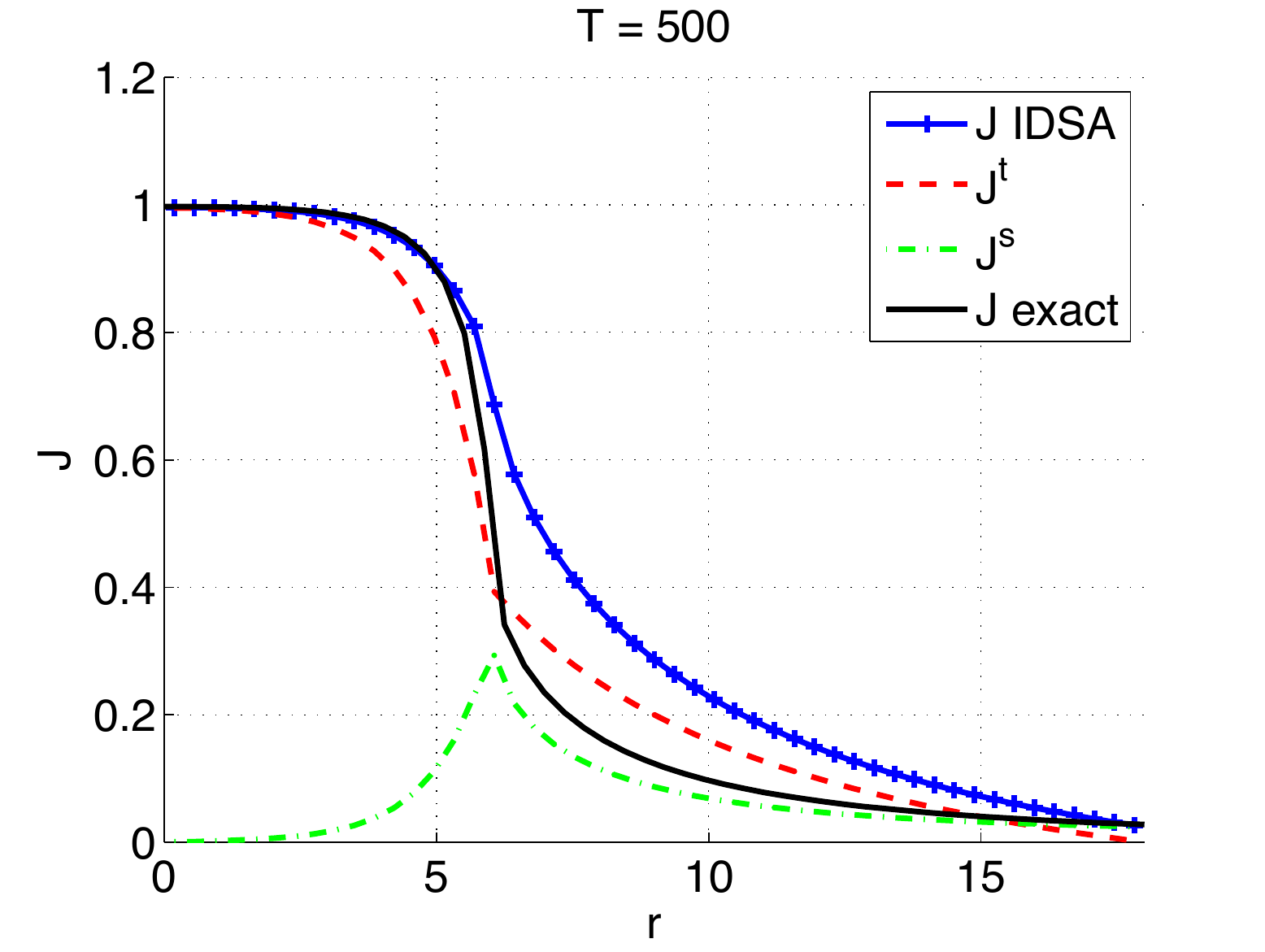}}
\quad
\subfigure[Proportion of trapped $h^t=\frac{J^t}{J^t+J^s}$ and streaming
particles $h^s=\frac{J^s}{J^t+J^s}$ at $T = 500$. ]{\label{Fig:I-4-SpTr.d}\includegraphics[width=.42\textwidth]{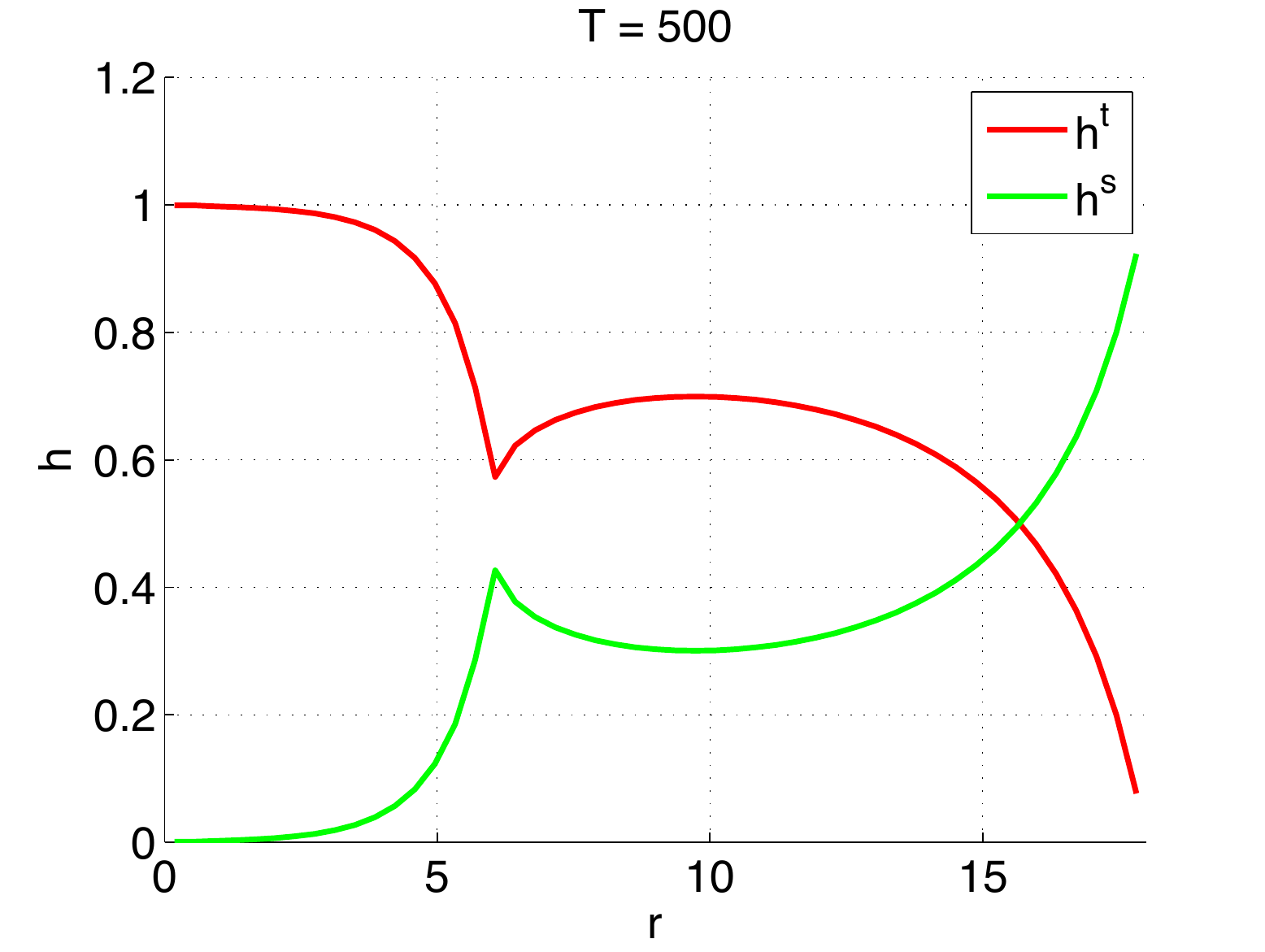}}

\subfigure[Solution at time $T = 1000$.]{\label{Fig:I-4-SpTr.e}\includegraphics[width=.42\textwidth]{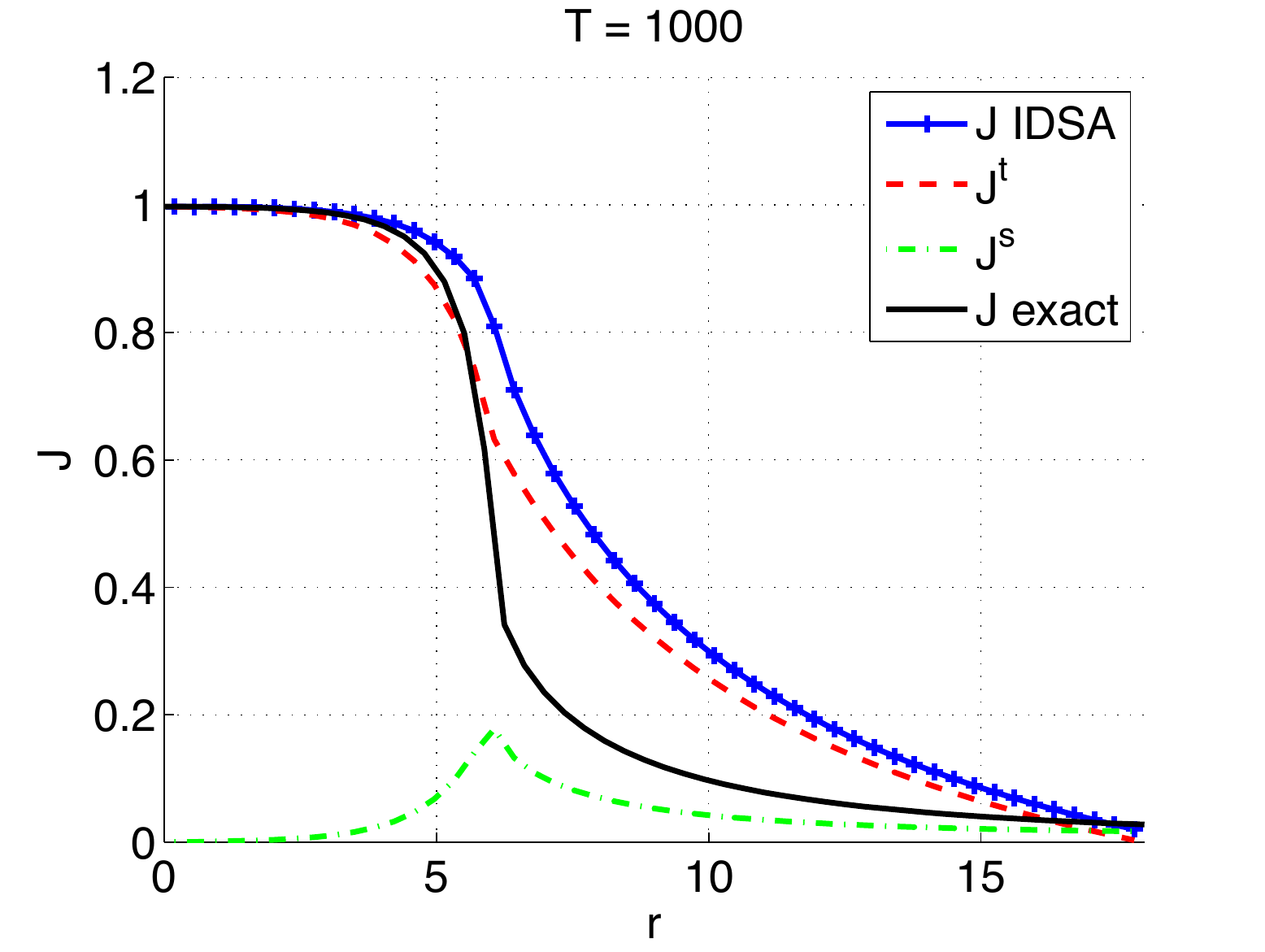}}
\quad
\subfigure[Proportion of trapped $h^t=\frac{J^t}{J^t+J^s}$ and streaming
particles $h^s=\frac{J^s}{J^t+J^s}$ at $T = 1000$. ]{\label{Fig:I-4-SpTr.f}\includegraphics[width=.42\textwidth]{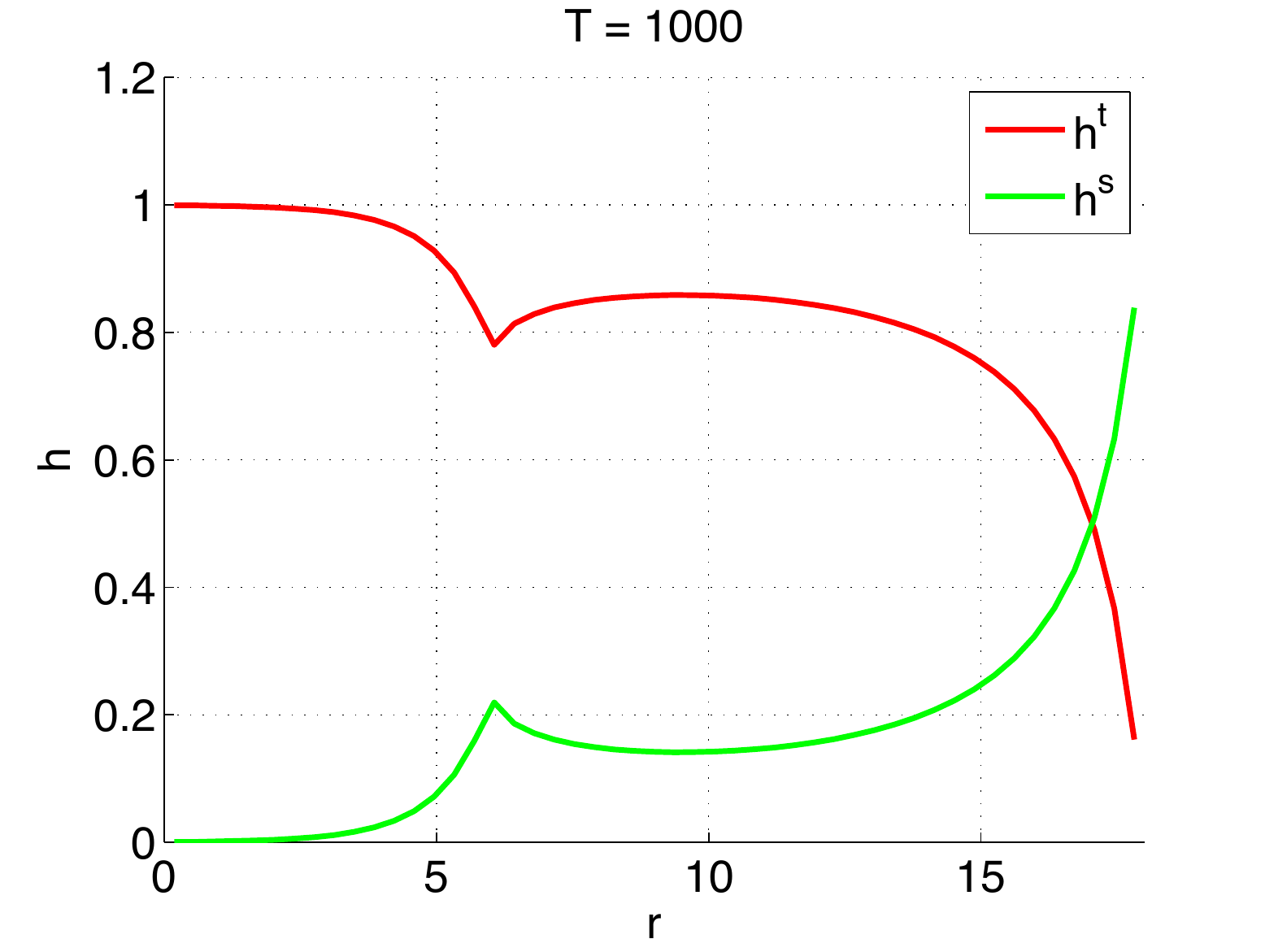}}

\caption{Results for the IDSA with small opacity $\kappa_{\rm a} = 10^{-3}$ for
  $r\geq 6$. The different panels show the evolution of the
  distributions. The black curve displays the solution of the
  homogeneous sphere. The right panels illustrate the takeover of
  trapped particles in the streaming region.}
\label{fig:SpuriousTrapped}
\end{figure}
The first problem that we discuss here is referred to as the \emph{Spurious
  Trapped} problem. This is a problem in modeling that occurs in regions with
small but non zero opacities ($\kappa$ small). The problem originates in the definition
of the coupling term $\Sigma$. Such a region should be transparent and
therefore described by the free streaming limit. Hence, the IDSA
limiter should always choose the 
free streaming regime in this case. We now construct an example in
which the definition of diffusion source $\Sigma$ leads to a wrong
regime. 
\begin{proposition}\label{prop:SpuriousTrapped} In a transparent
  region where $\kappa_{\rm a}= \mathcal
  O(\varepsilon)$ and where the distribution $J^t$ is
  locally constant and $J^s < B$, we have
\begin{equation}
\Sigma \neq \kappa_{\rm a}B,
\end{equation}
that is the IDSA does not choose the free streaming regime in this case.
\end{proposition}
\begin{proof}
The condition on $\kappa_{\rm a}$ is needed in order to have a
transparent region, for more details about this asymptotic limit, see
\cite{BeFrGaLiMiVa12big, MichaudPHD}. We now compute the value of the diffusion
source in this case.
\begin{equation}
\begin{aligned}
\Sigma &= \min\left \{
  \max\left[-\frac{1}{r^2}\partial_r\left(\frac{r^2}{3\kappa}\partial_r
      J^t\right)+\kappa_{\rm a}J^s ,0\right],\kappa_{\rm
    a}B\right\}\\
&=\min\left \{
  \max\left[\kappa_{\rm a}J^s ,0\right],\kappa_{\rm a}B\right\},
\end{aligned}
\end{equation}
where we used the fact that $J^t$ is locally constant.
We now observe that $\kappa_{\rm a}>0$ and that $J^s<B$ by hypothesis,
this leads to 
\begin{equation}
\Sigma = \kappa_{\rm a}J^s < \kappa_{\rm a}B.
\end{equation}
This concludes the proof.
\end{proof}
We now discuss the relevance of the hypotheses used in Proposition
\ref{prop:SpuriousTrapped}. As we have 
mentioned in the proof, the assumption on $\kappa_{\rm a}$ just
reflects the fact that we are in a transparent region. In such a
region, Eq. \eqref{eq:tr.free2} implies that the trapped distribution
vanishes. Finally, the condition
on $J^s$ is natural in radiative transfer problems. In fact, the
equilibrium distribution is mainly driven by emission and absorption
processes which are small in a transparent region. The transport
processes therefore tend to reduce the distribution which leads to
$J^s<B$. 

Now that we know that the hypotheses can be realized, we look at
the resulting dynamics of the trapped particles
\begin{equation}
\partial_t J^t = \kappa_{\rm a}(B-J^t-J^s).
\end{equation}
Because $B-J^s>0$, we can assume that there exists $B'$ such that
$B-J^s\geq B' >0$ and 
\begin{equation}
J^t > B'(1-{\rm e}^{-\kappa_{\rm a} t})
\end{equation}
that is, the trapped particle distribution is growing.

\begin{remark} Note that if $\kappa_{\rm a} = 0$ in the free streaming
  regime, a non-vanishing trapped distribution will not be eliminated
  because its evolution becomes
\begin{equation}
\partial_t J^t_{\rm free} = 0.
\end{equation}
and the distribution, whatever it is at first, is stationary.
\end{remark}

Now that we have seen that the spurious trapped problem may occur, we
show a numerical illustration of this problem. For the discretization,
we used the discretization proposed by Liebend\"orfer in
\cite{LiebendoerferEtAl09big} described in more detail in
\cite{BeFrGaLiMiVa12_esaimbig,MichaudPHD}. For the grid parameters, we choose
$N_r = 50$ and $\Delta t = 0.1$. For the parameters of the radiative
transfer equation, we choose $B=1$, $\kappa_{\rm s}=0$ and
\begin{equation}
\kappa_{\rm a}= \left \{
\begin{array}{ll}
1,& \text{if }r<R = 6,\\
\varepsilon = 10^{-3},&\text{if }r\geq R = 6.
\end{array}\right.
\end{equation}
\begin{figure}[h]
\centering
\includegraphics[width = .44\textwidth]{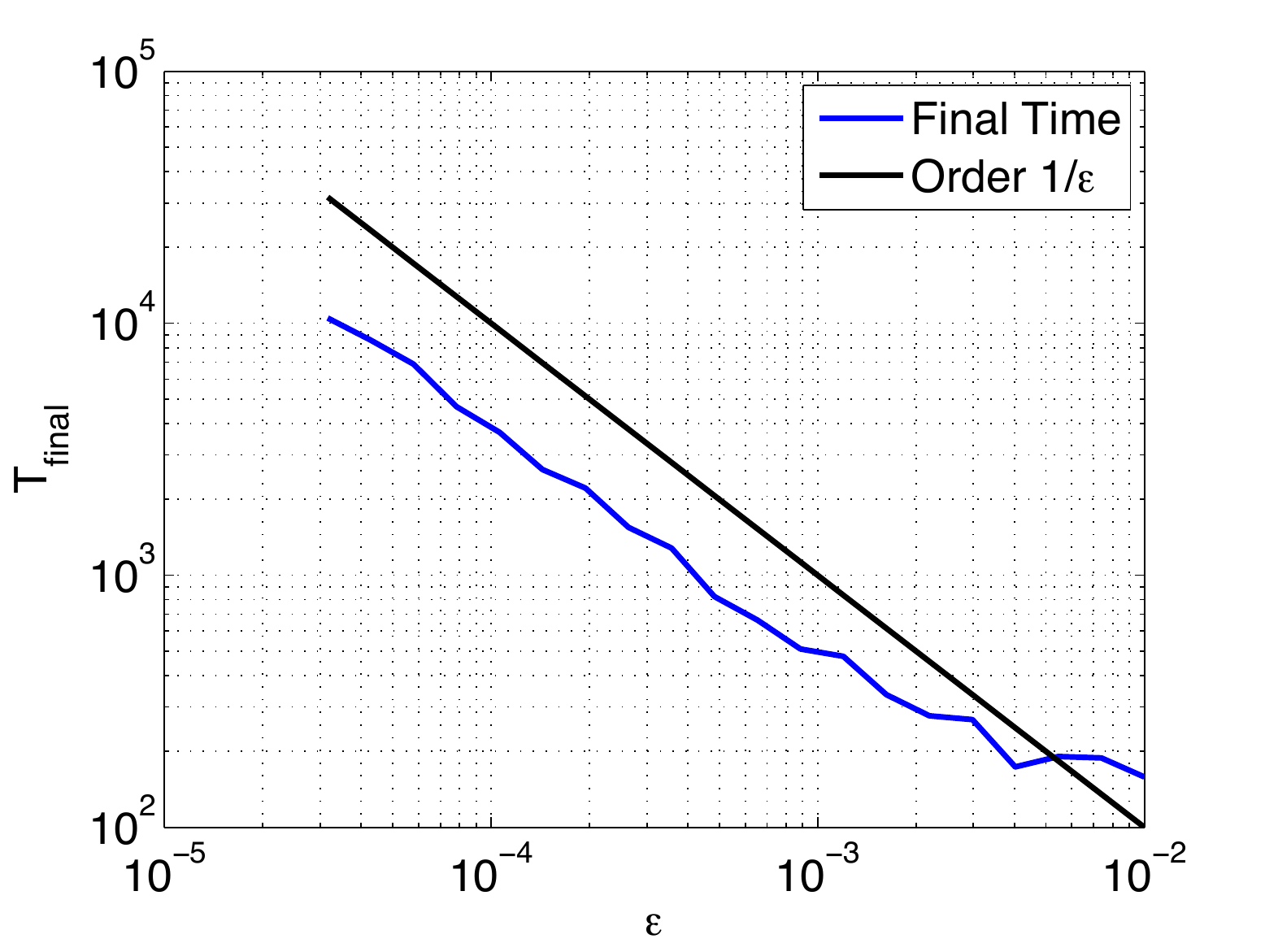}
\caption{Time needed to reach a stationary distribution dominated by
  spurious trapped particles with respect to $\varepsilon$.}
\label{fig:SpTrappedOrder}
\end{figure} 

The results are displayed in Figure~\ref{fig:SpuriousTrapped}. In this
Figure, we also show the proportions $h^t=\frac{J^t}{J^t+J^s}$ and
$h^s=\frac{J^s}{J^t+J^s}$ of trapped and streaming particles on the
right panels. This illustrates the fact that the trapped component is
growing. This problem develops on a slow timescale, and becomes visible
only on a long timescale. When $\varepsilon$ goes to zero, the time needed
to create spurious trapped particles tends to infinity and this problem
disappears. In Figure \ref{fig:SpTrappedOrder}, we display the time
needed to reach a stationary distribution dominated by spurious
trapped particles in the streaming region. The results show that the
time needed to reach a stationary state grows roughly with an order of
$\mathcal O(\varepsilon^{-1})$. Numerically, a linear regression on
the results presented in Figure \ref{fig:SpTrappedOrder}, where we
excluded the $5$ first points because for too big $\varepsilon$ the
free streaming limit is no longer a good approximation,
gives an order of $\mathcal O(\varepsilon^{-0.90})$. It is reasonable
to think that inside a CCSN, the time interval over which the solution
is computed is too short for this problem to develop.

\subsubsection{Instability of the coupling}

\begin{figure}[h]
\centering

\subfigure{\label{Fig:I-4-InDevE.a}\includegraphics[width=.42\textwidth]{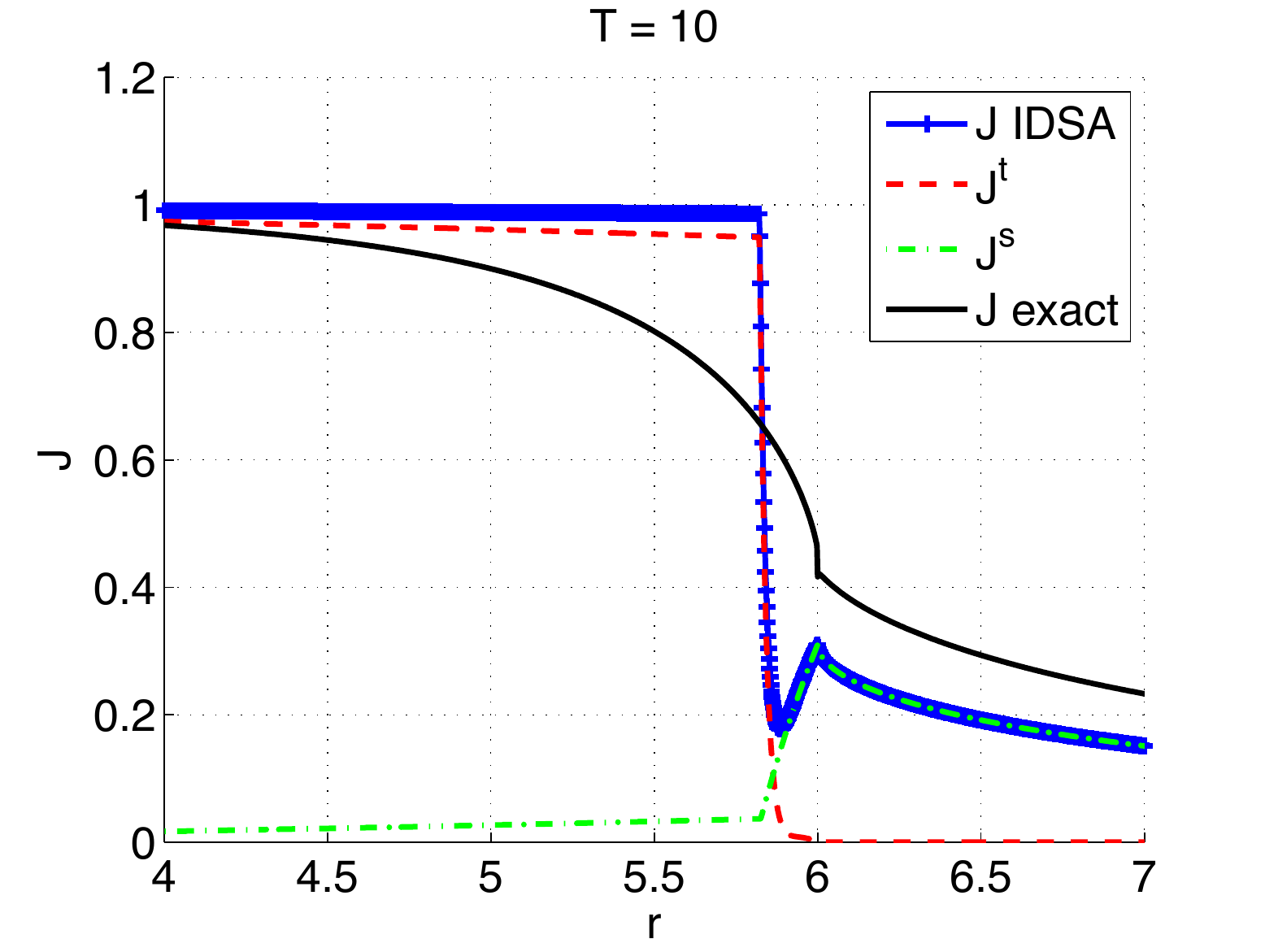}}
\quad
\subfigure{\label{Fig:I-4-InDevE.b}\includegraphics[width=.42\textwidth]{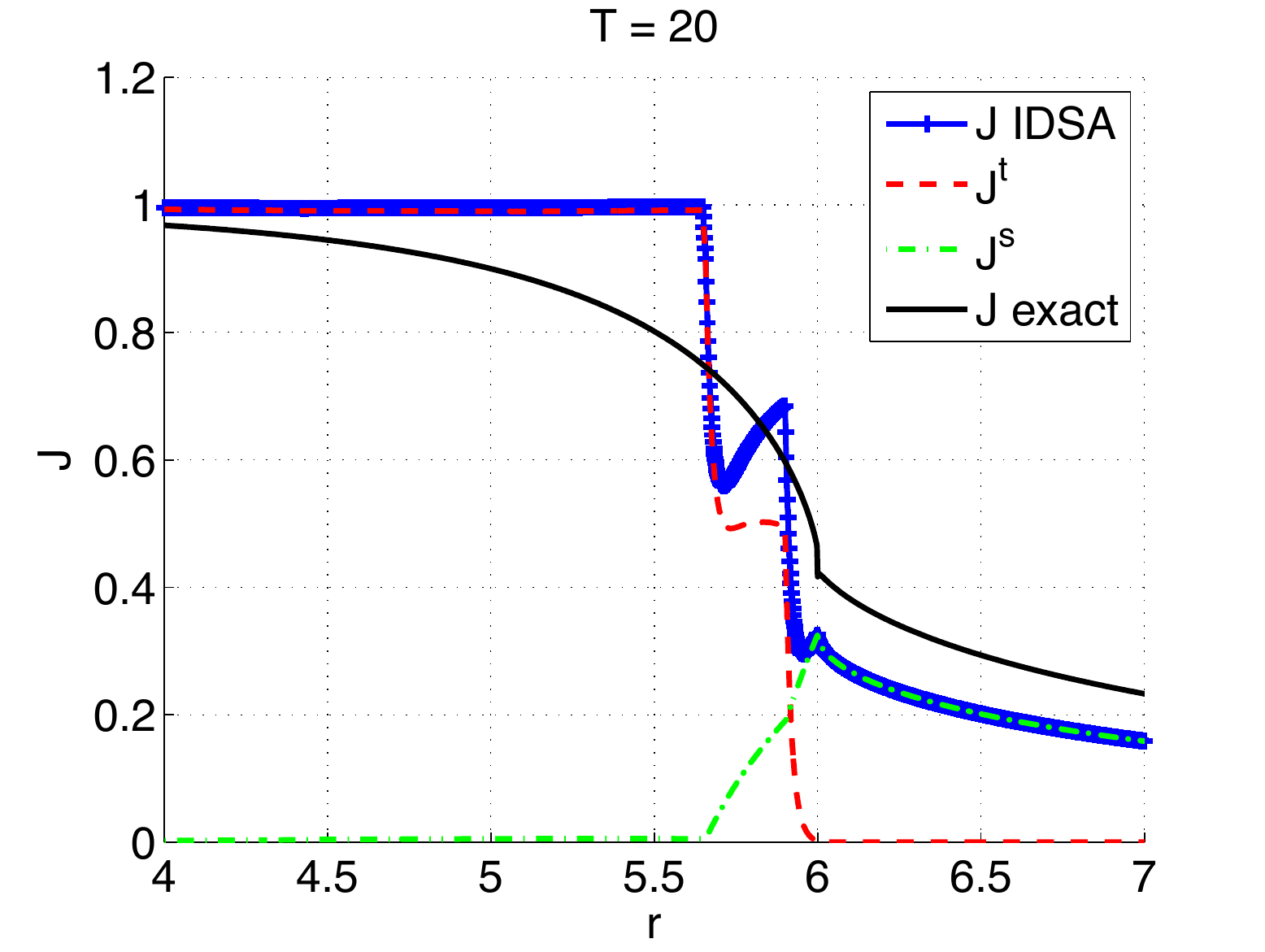}}
\hfill

\subfigure{\label{Fig:I-4-InDevE.c}\includegraphics[width=.42\textwidth]{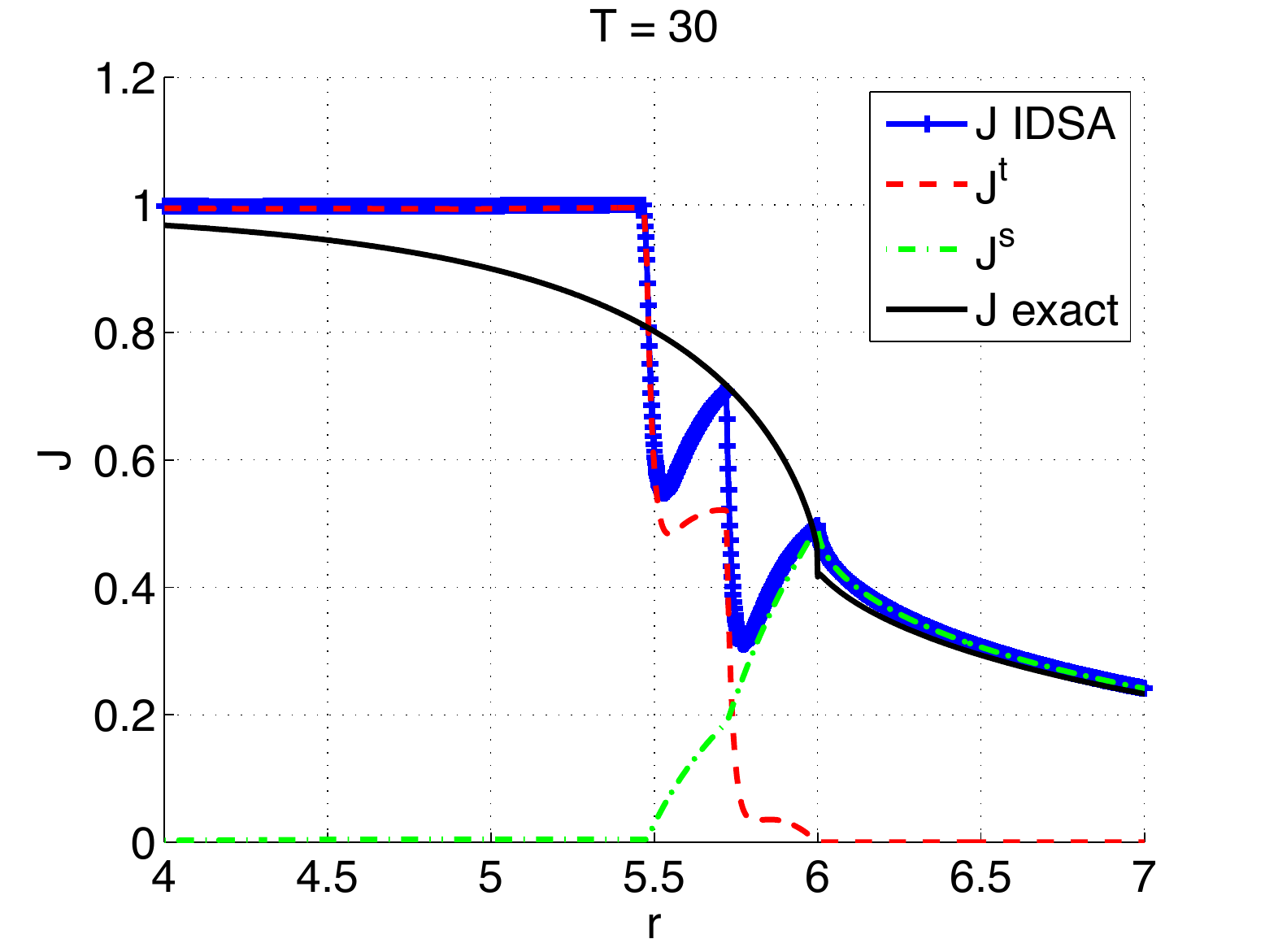}}
\quad
\subfigure{\label{Fig:I-4-InDevE.d}\includegraphics[width=.42\textwidth]{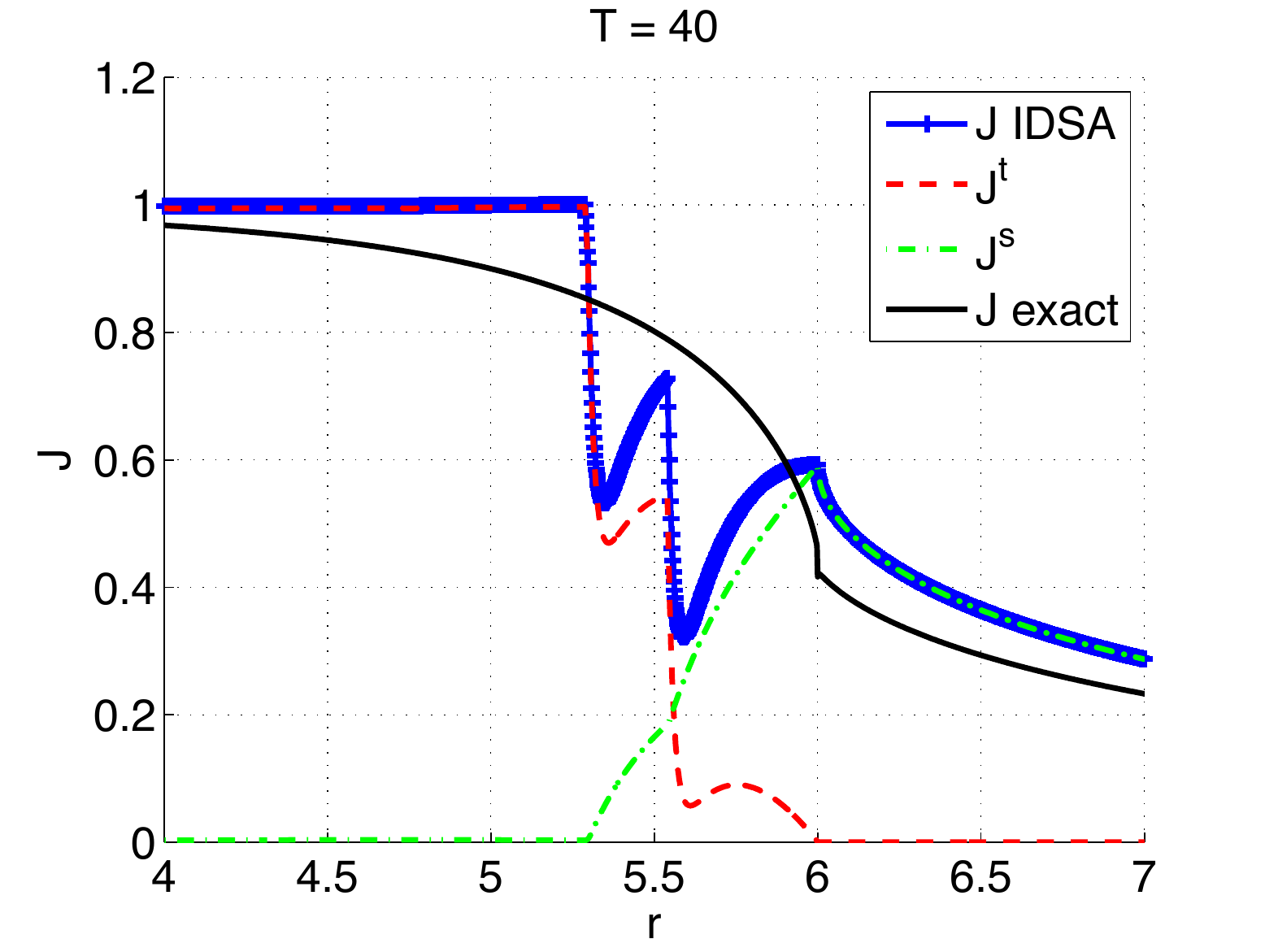}}
\hfill

\subfigure{\label{Fig:I-4-InDevE.e}\includegraphics[width=.42\textwidth]{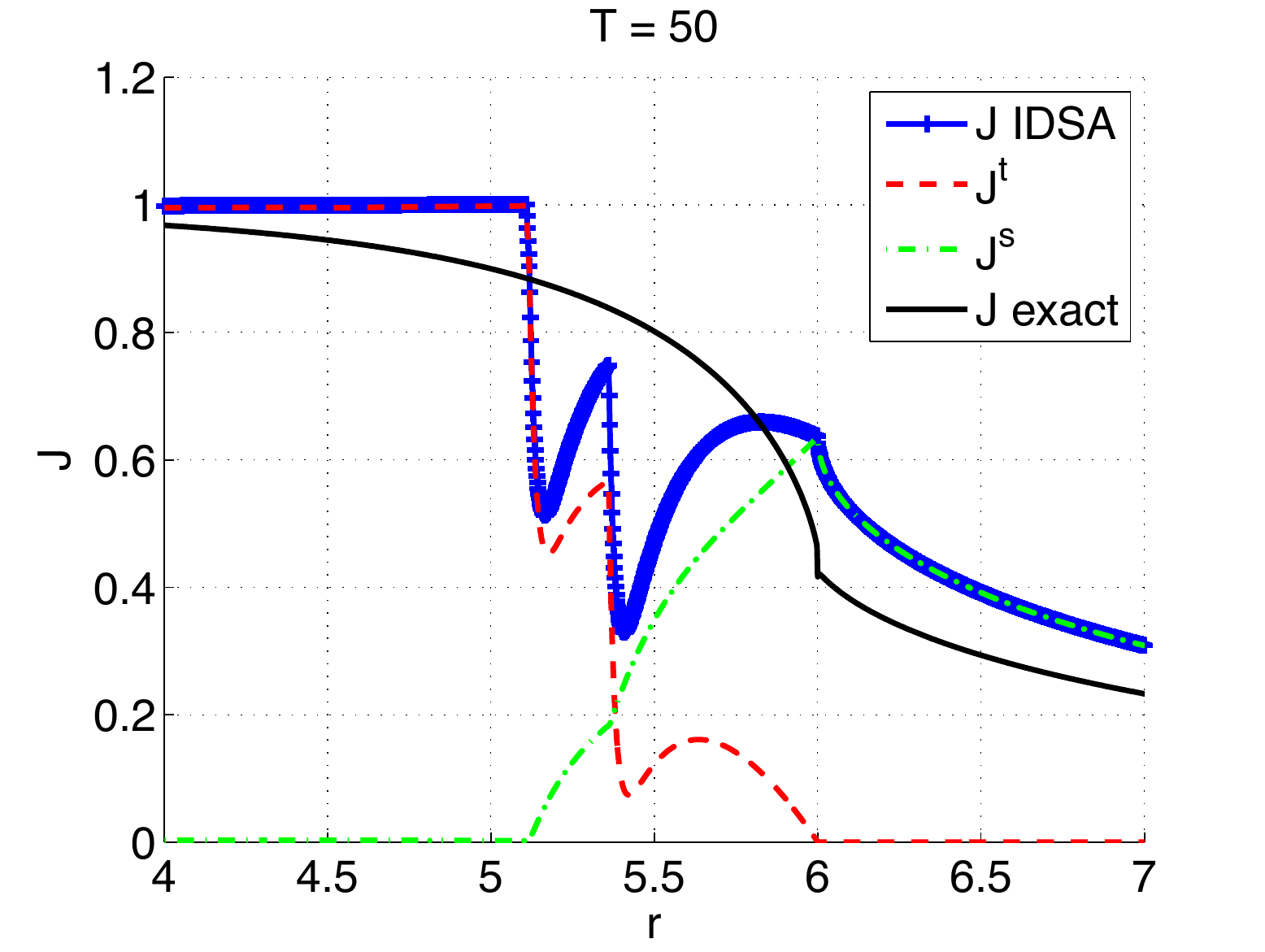}}
\quad
\subfigure{\label{Fig:I-4-InDevE.f}\includegraphics[width=.42\textwidth]{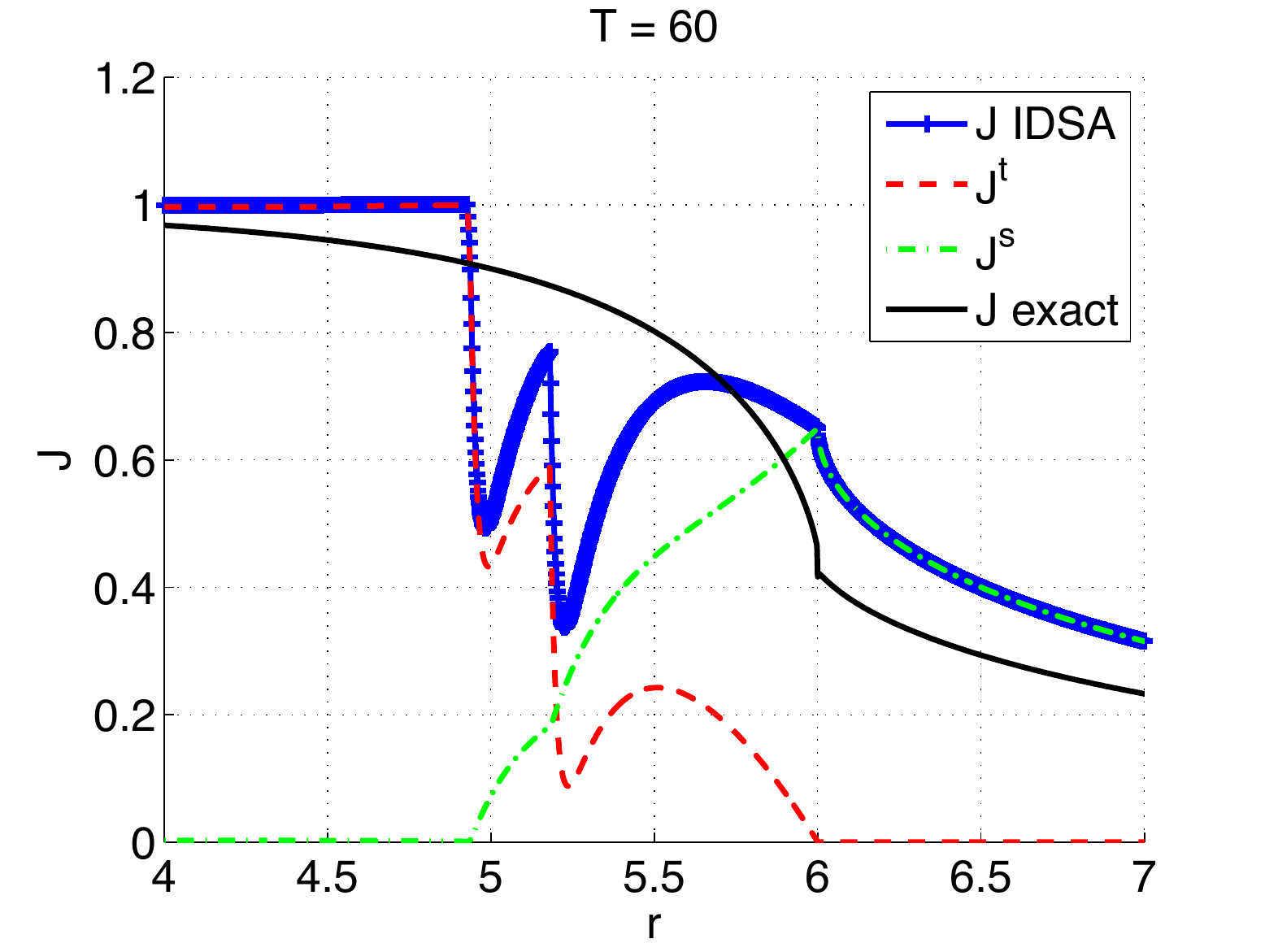}}
\hfill

\subfigure{\label{Fig:I-4-InDevE.g}\includegraphics[width=.4\textwidth]{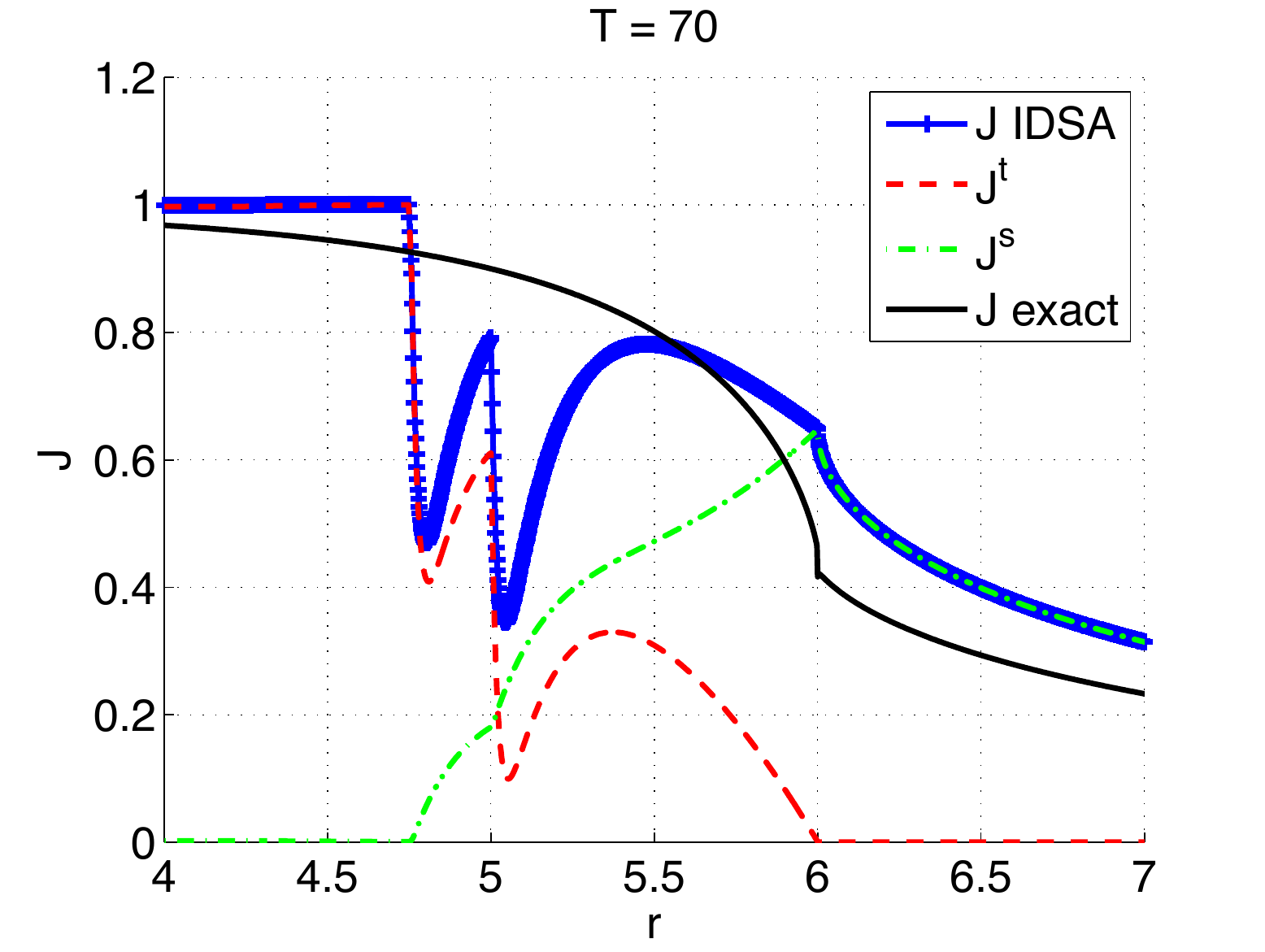}}
\quad
\subfigure{\label{Fig:I-4-InDevE.h}\includegraphics[width=.4\textwidth]{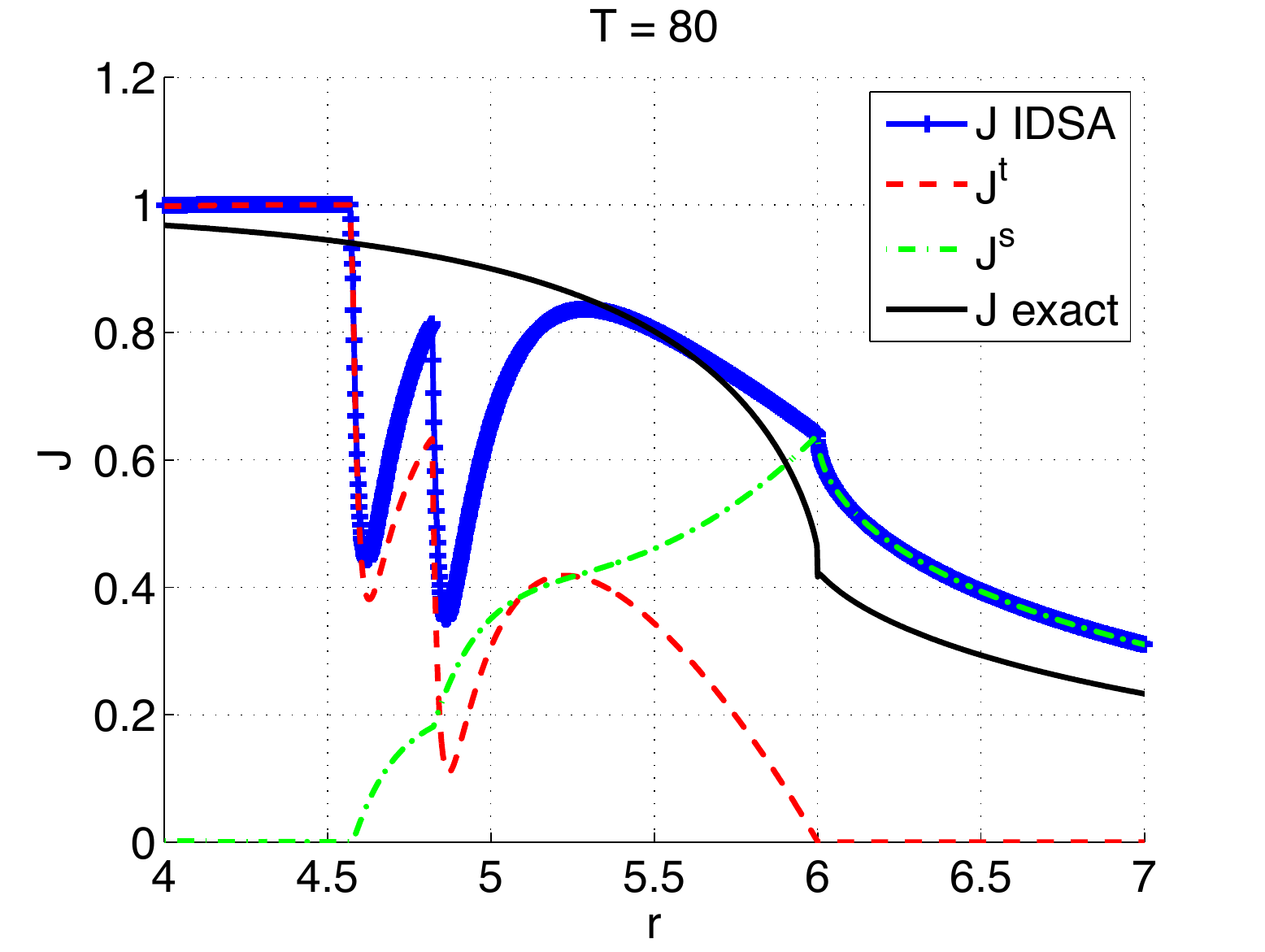}}
\hfill

\caption{Development of the instability. These results have been obtained
with an homogeneous sphere example defined by $\kappa=
\mathds{1}_{r<6}$. The figure illustrates the formation of the instability
and its inward propagation.}
\label{fig:InstabilityEarly}
\end{figure}

\begin{figure}[h]
\centering

\subfigure{\label{Fig:I-4-InDevL.a}\includegraphics[width=.42\textwidth]{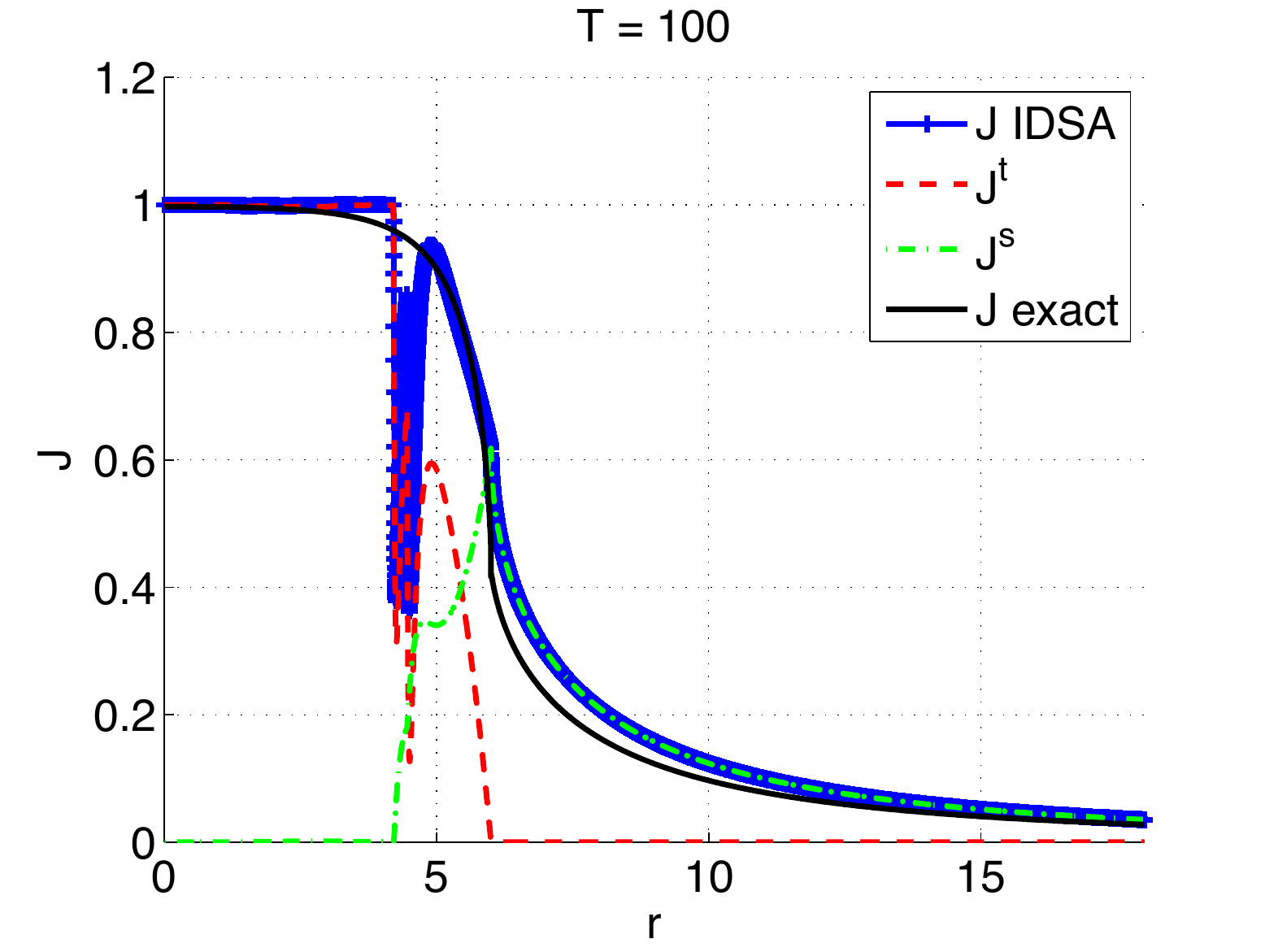}}
\quad
\subfigure{\label{Fig:I-4-InDevL.b}\includegraphics[width=.42\textwidth]{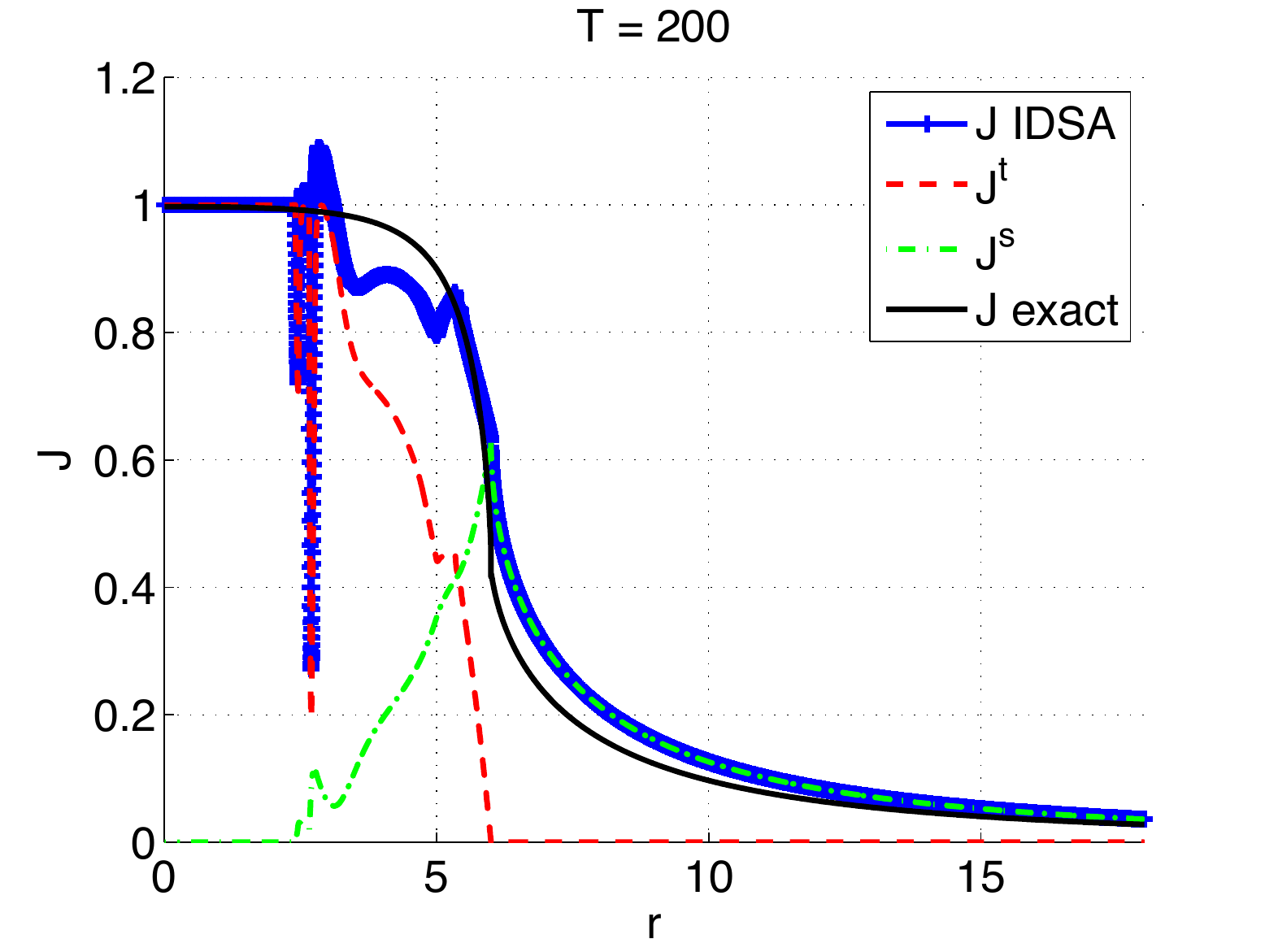}}
\hfill

\subfigure{\label{Fig:I-4-InDevL.c}\includegraphics[width=.42\textwidth]{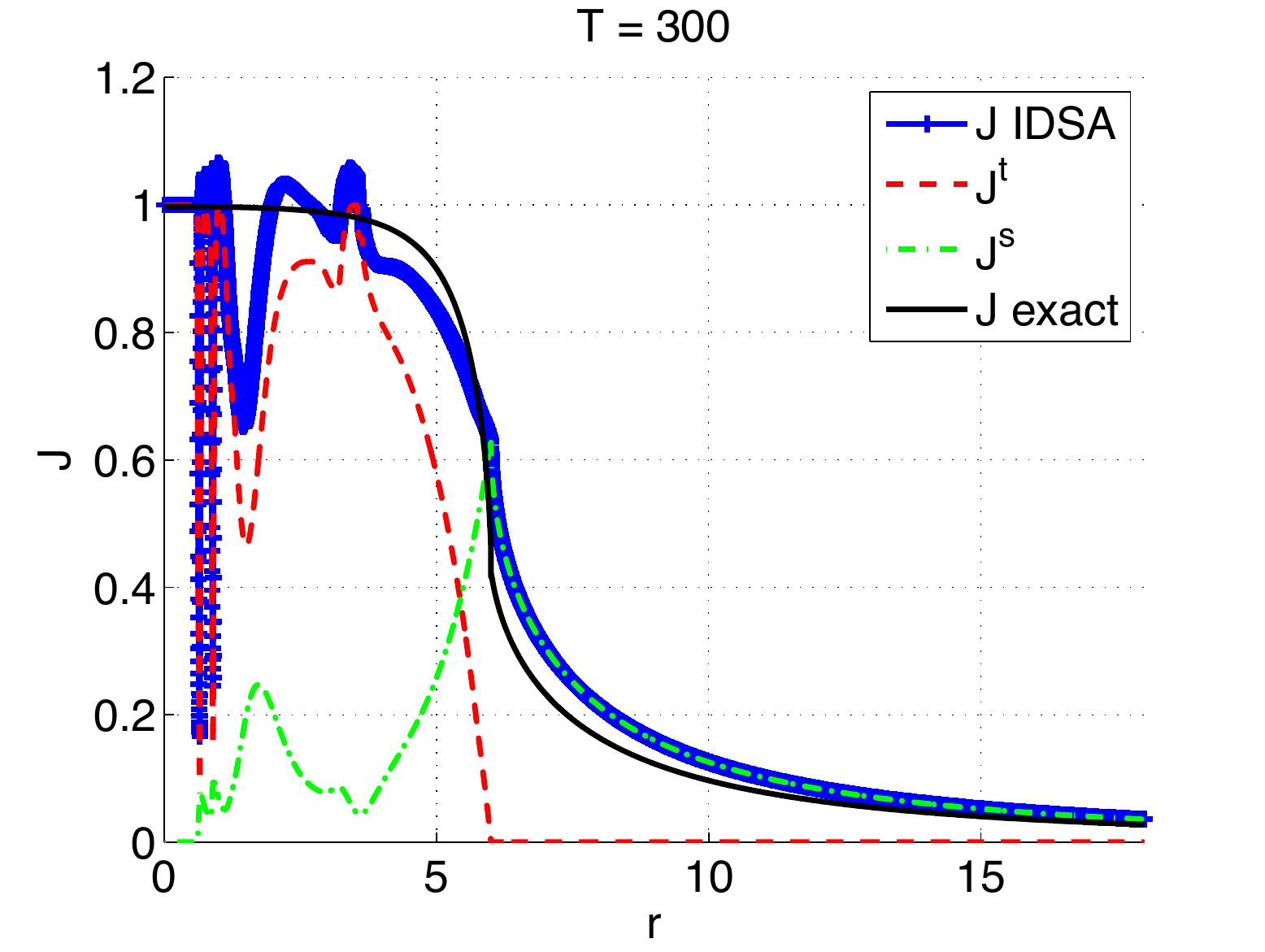}}
\quad
\subfigure{\label{Fig:I-4-InDevL.d}\includegraphics[width=.42\textwidth]{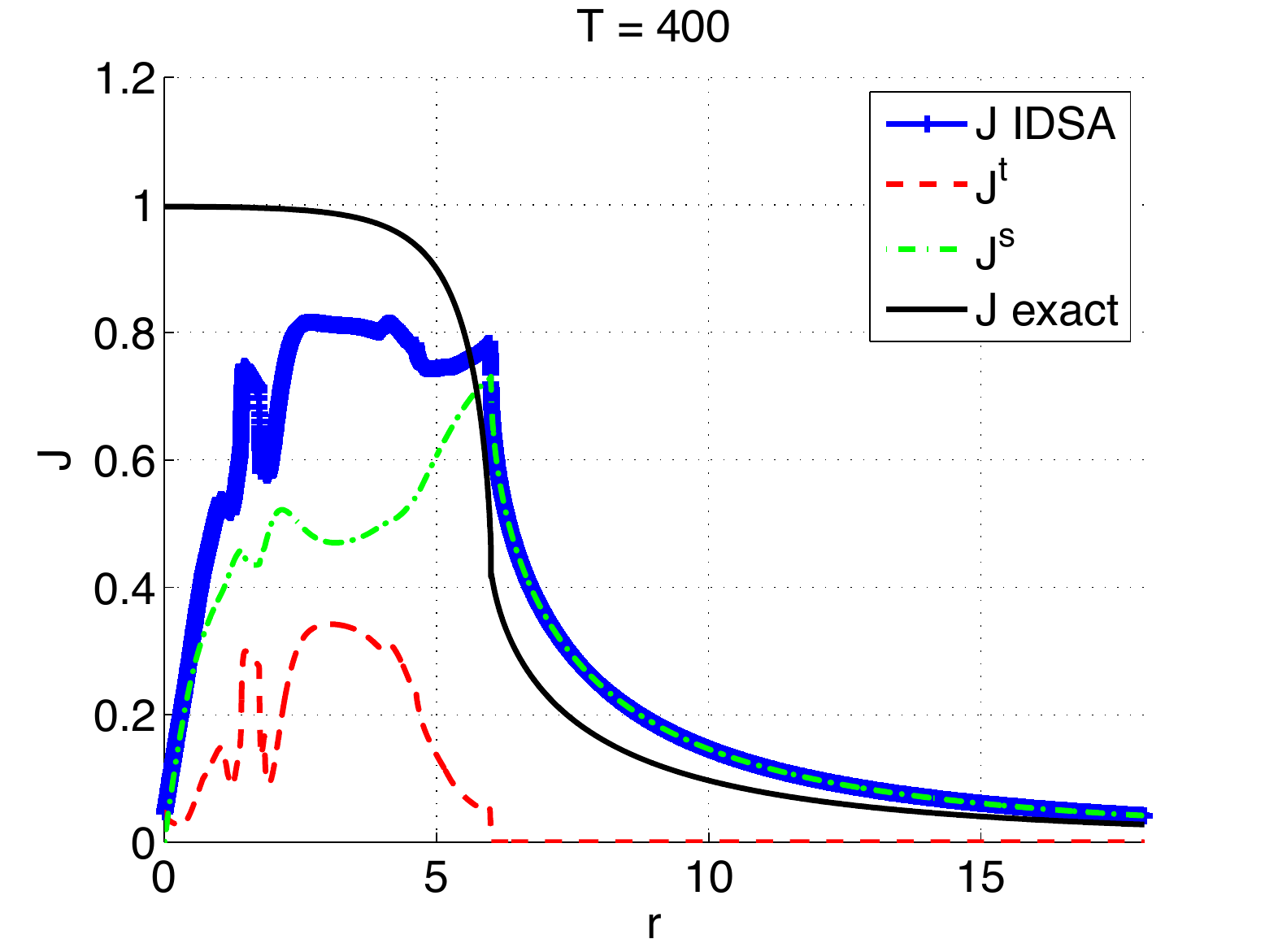}}
\hfill

\subfigure{\label{Fig:I-4-InDevL.e}\includegraphics[width=.42\textwidth]{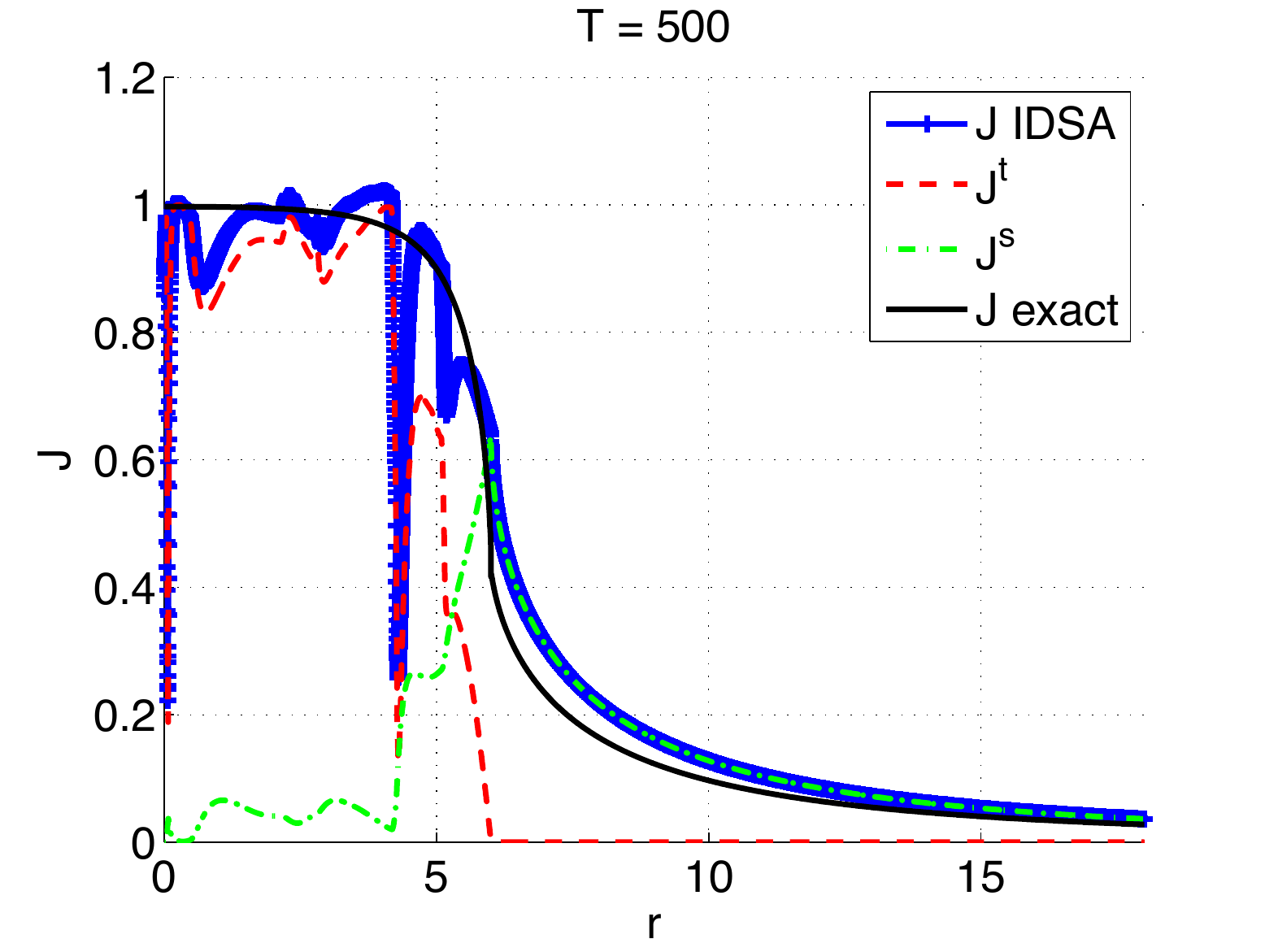}}
\quad
\subfigure{\label{Fig:I-4-InDevL.f}\includegraphics[width=.42\textwidth]{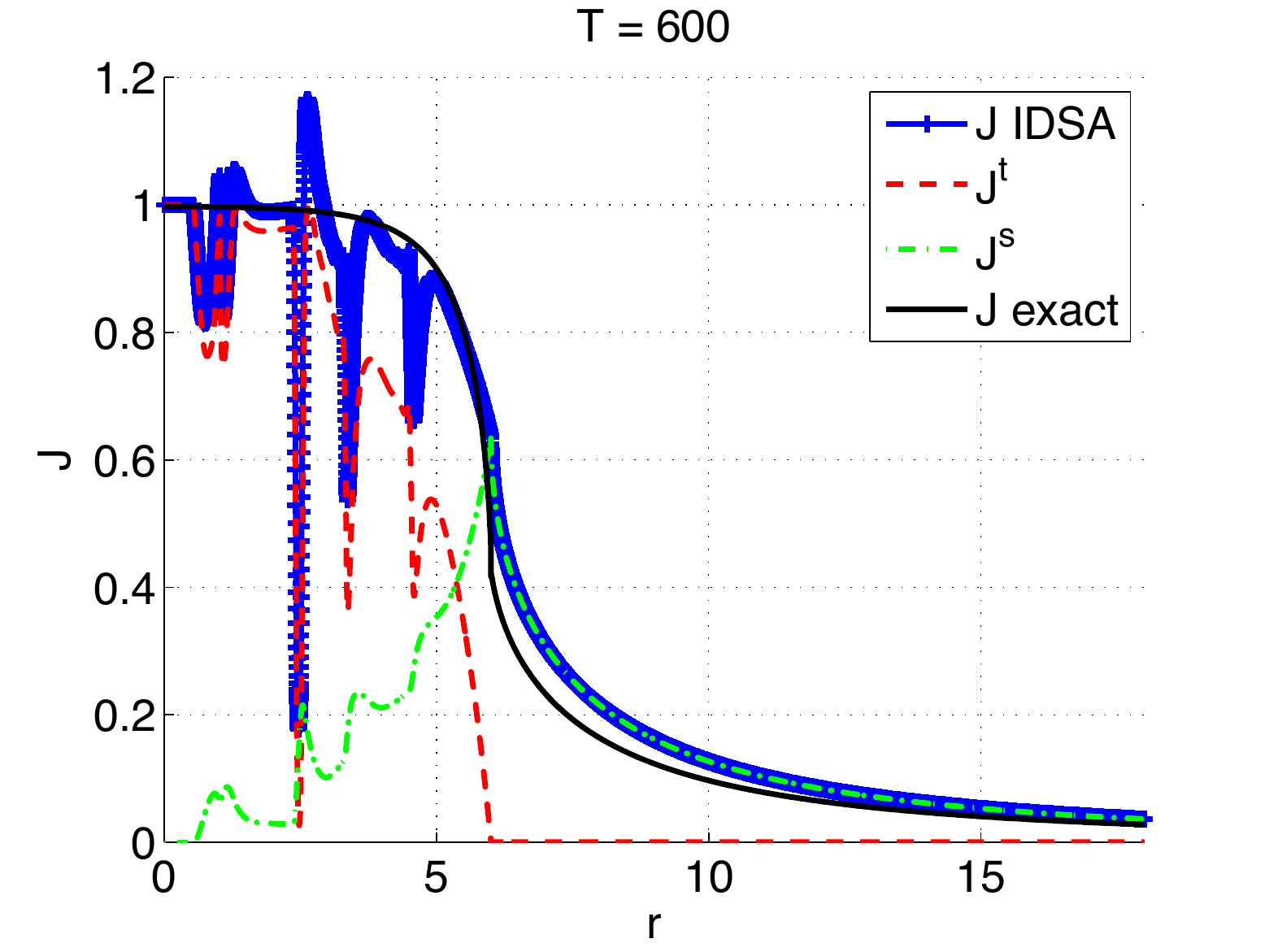}}
\hfill

\subfigure{\label{Fig:I-4-InDevL.g}\includegraphics[width=.42\textwidth]{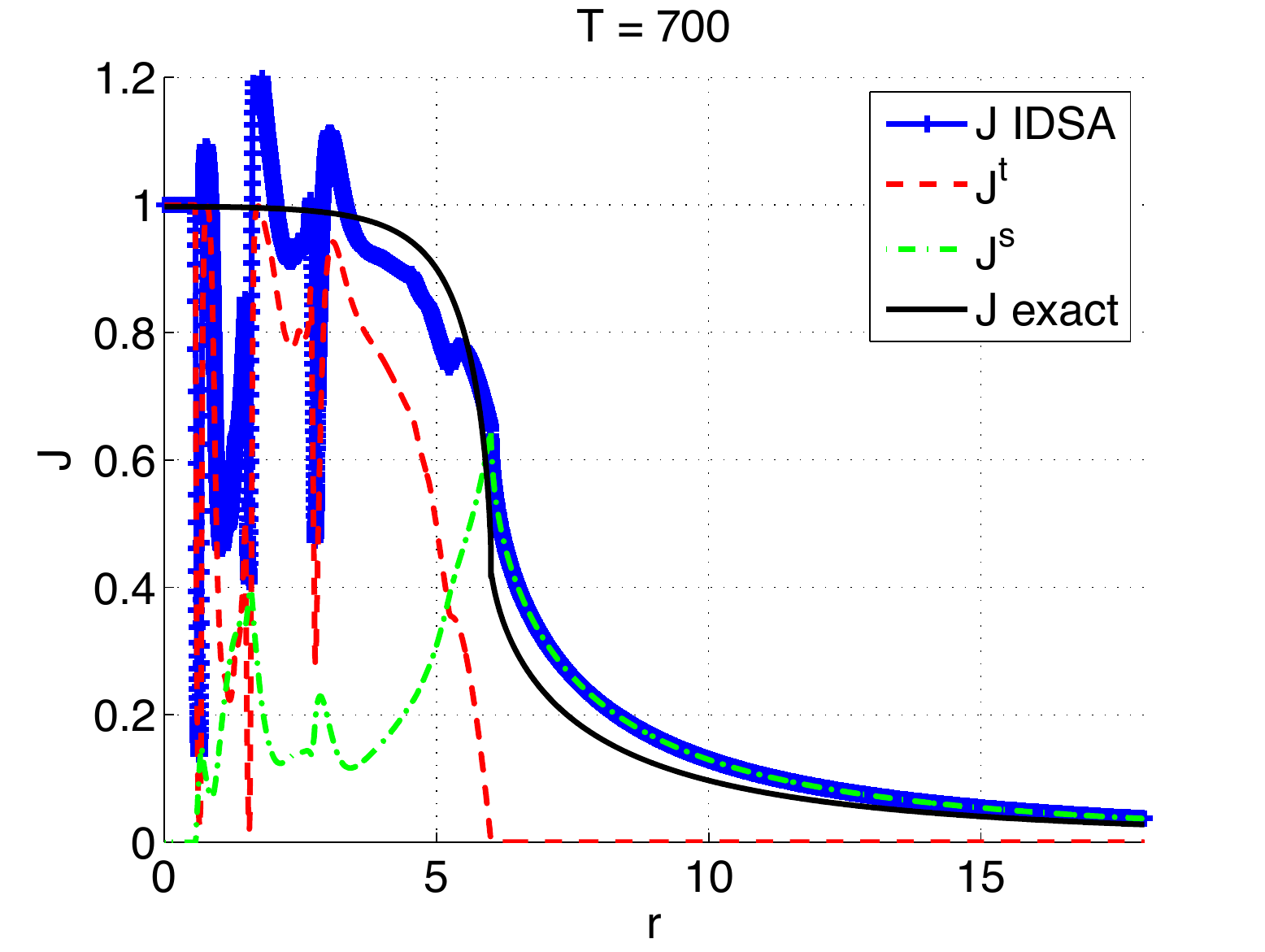}}
\quad
\subfigure{\label{Fig:I-4-InDevL.h}\includegraphics[width=.42\textwidth]{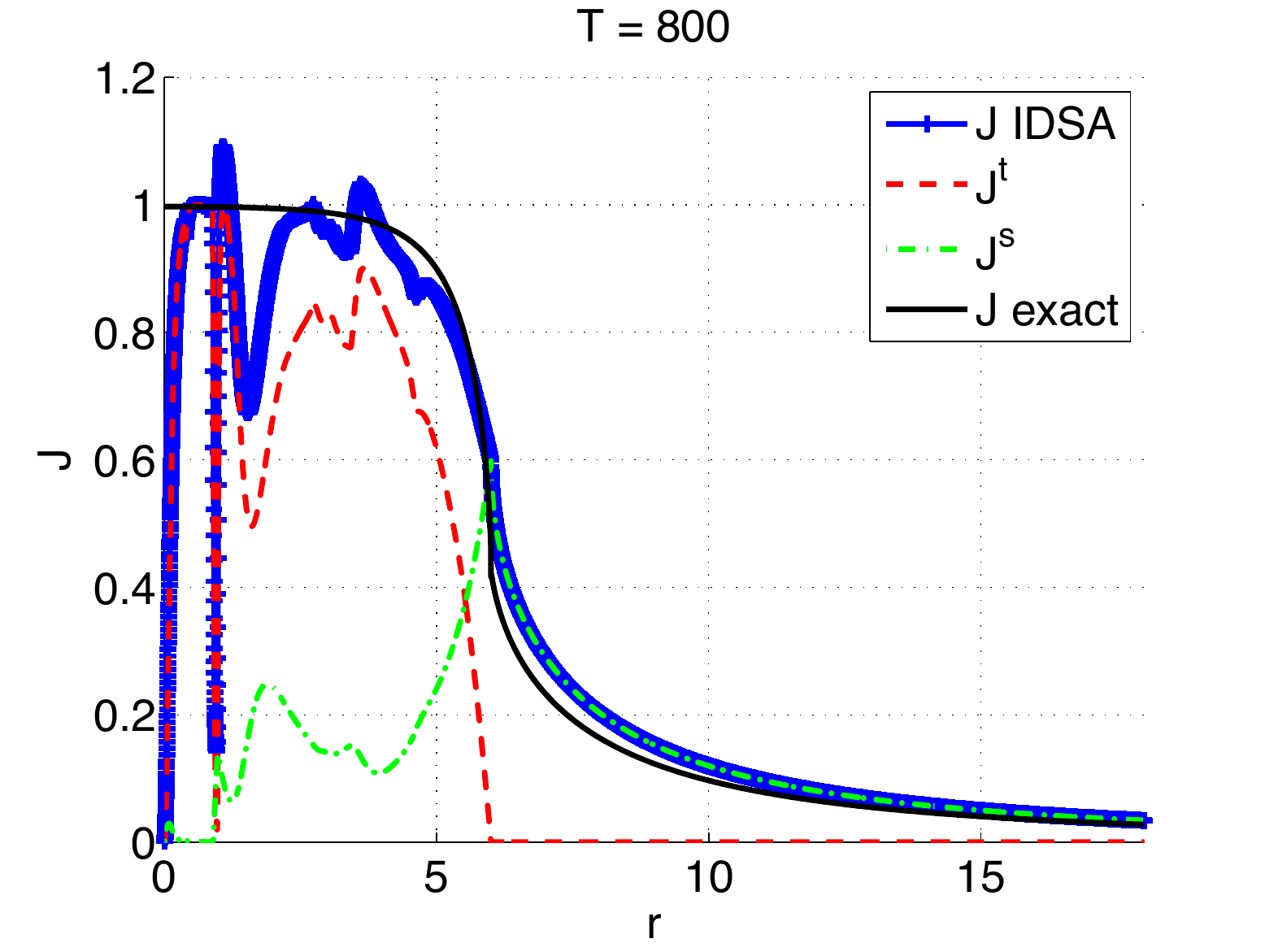}}
\hfill

\caption{Long time behavior of the instability. These results have been obtained
with a homogeneous sphere example defined by $\kappa =
\mathds{1}_{r<6}$. The figure illustrates the long time behavior of the solution.}
\label{fig:InstabilityLate}
\end{figure}
The second problem we discuss is also linked with the diffusion source
$\Sigma$. It takes the form of an instability that develops at a
boundary between an opaque and a transparent region. It originates in
the fact that the diffusion term in $\Sigma$ can become large for two
reasons: the first is because $\kappa$ is small and the domain is
transparent, and this triggers a correct transition to the streaming
limit; the second is because the concavity becomes large. At
the boundary between a very opaque and a transparent region, the
second case may occur leading to a switch to the streaming regime in a
region 
where the opacity $\kappa$ is large. That is, where the diffusion
limit is valid and the diffusion regime of the IDSA should apply. As a
consequence, the IDSA virtual boundary between the opaque and
the transparent domains is shifted
inward. This shift is helped by the creation of streaming particles
that feed back into the dynamics of the trapped particles in the form of
an inward advection term, see Eq. \eqref{eq:tr.with.adv}. At some
point, the diffusion limit is recovered near the real 
boundary and starts to grow again as expected, but at the virtual boundary
of the IDSA, the trapped distribution continues to locally decrease leading
to a locally radially growing distribution of trapped particles which
is not physical, see Remark~\ref{rem:2}. As an effect, an instability
of the underlying model
is created and propagated inward.

As a numerical illustration of such an instability, we apply the IDSA
to the homogeneous sphere test case defined by $\kappa_{\rm s}(r) = 0$,
$\kappa_{\rm a}(r) =\mathds{1}_{r<6} $ and $B=1$. For the grid
parameters, we set $N_r=10^4$ and $\Delta t = 0.1$.

Figure \ref{fig:InstabilityEarly} displays the formation of the
instability near the boundary of the homogeneous sphere. We see for
example in the first subfigure that the virtual boundary of
the IDSA, where the trapped distribution vanishes, stands at a radius
$r<R$. This illustrates the shift of the interface discussed above. The
next subfigures show the early developement of the instability and
its inward transport due to the feedback of the streaming component
into the trapped component dynamics, see Eq.~\eqref{eq:tr.with.adv}.

The behavior of the solution on longer times is displayed in Figure
\ref{fig:InstabilityLate}. This figure shows the evolution of the
instability. The solution becomes more complicated as time
passes. Note that the solution remains bounded and never explodes like a
usual nu\-me\-ri\-cal instability due to numerical problems such as a
CFL condition. This instability in the 
solution comes from the modeling and the definition of the coupling
term $\Sigma$. 
\begin{remark}\label{rem:4} The instability is linked with the size of the
  diffusion term. Note that this quantity also depends on the
  discretization as the second derivative will be of the order of
  $\mathcal O((\Delta r)^{-2})$. This means that if the discretization
  is coarse enough, then the instability does not occur. This explains
  the fact that we used a coarse grid in the spurious trapped
  numerical experiment.
 In fact, if we
  refine the grid, the spurious trapped numerical experiment will also
  show the numerical instability that we just described.
\end{remark}

We have shown with these two examples that the coupling in the IDSA
equations can lead to wrong solutions that are due to the application of a
wrong asymptotic regime in parts of the domain. This can lead to the
creation of spurious trapped particles or to the creation of an
instability in the solution. We have also discussed the influence of the
advection term coming from the coupling with streaming particles in
the trapped diffusion equation. It seems that this term helps the
developement of the instability.

A natural question to ask is: why did the IDSA work in the case of the
CCSN models performed in \cite{LiebendoerferEtAl09big}? We answer this
question below.

\subsection{Why does the IDSA work in practice?}
As we mentioned in Section \ref{sec:Intro}, in the numerical tests
performed by Liebend\"orfer et al. in \cite{LiebendoerferEtAl09big},
the IDSA is quite accurate and good enough to be used in
simulations. We just discussed two important mathematical issues of this
approximation. How is it possible that the IDSA works in practice when
it fails on simple examples? Part of the answer has been given in
Remark \ref{rem:4}. If the discretization is coarse enough, the
instability can not develop. In the application of the IDSA to the
CCSN modeling, the number of grid points is $N_r=1000$ and the
repartition of points is such that the instability of the underlying
model does not
develop. However, it is probable that a refinement of the radial grid will
lead to the apparition of an instability in the IDSA. 
Note that in a real CCSN simulation, the reactions in the center become very large and as a result, the instability is quickly damped. In our example for the instability, we used an artificially small opacity $\kappa = 1$. For larger opacities, the instability is damped and vanishes in the limit of the infinitely opaque homogeneous sphere.
Another possible
explanation 
is that in the CCSN numerical experiments, the radiative transfer equation is
coupled to the background matter. This coupling is done through the
trapped particles and their distribution is reconstructed at every
time step with a smoothing effect that may help to get rid of the
instability.
Regarding the spurious trapped particles, we have observed that the
growing of the spurious trapped particles component is slow and that
it takes a long time to dominate the dynamics, if ever. In fact, in a
CCSN simulation, the time window on which the neutrino radiative
transfer needs to be computed is short and the spurious trapped
particles do not contribute significantly to the dynamics.

We now study in full detail the case of the homogeneous sphere where
an ana\-ly\-ti\-cal solution is known. This idealized example will shed
light on the coupling mechanism and will help us to partly correct or avoid
the problems we have des\-cribed.

\section{The case of the stationary homogeneous sphere}\label{sec:HS}
In this example we want
to solve the monochromatic equation \eqref{eq:RadTransp} with
$\kappa_{\rm s}=0$, that is
\begin{equation}\label{RadTransp}
\partial_t f + \mu \partial_rf+\frac{1-\mu^2}{r}\partial_\mu f=
\kappa_{\rm a}(B-f).
\end{equation}
We assume that the equilibrium distribution $B$ is constant. The
homogeneous sphere setting is completed by requiring a specific form
for $\kappa_{\rm a}$. We set 
\begin{equation}\label{eq:kappaHS}
\kappa_{\rm a} = \kappa \mathds{1}_{r<R}\,,
\end{equation}
where $\kappa$ is a constant and $\mathds{1}$ is the characteristic
function.
This example has been chosen for two main reasons. The first is
because it has a closed form solution that is for example
derived in Chapter 2 of Duderstadt and Martin \cite{Duderstadt79}. The
second reason is that this example naturally has two different regions
where the limit dynamics correspond to the limits used in the
derivation of the IDSA. It is therefore a very good and natural case
study for the IDSA.

 The steady state radiative
transfer equation
\eqref{RadTransp} has an analytical solution given by
\begin{equation}\label{eq:AnaSolHS}
f(r,\mu) = B(1-{\rm e}^{-\kappa s(r,\mu)}),
\end{equation} 
where 
\begin{equation}
s(r,\mu) = \left\{
\begin{array}{llr}\displaystyle
r\mu + RG(r,\mu)&r<R,&-1<\mu<1,\\
2RG(r,\mu)&r\geq
R,&\displaystyle\left[1-\left(\frac{R}{r}\right)^2\right]^{1/2}<\mu<1,
\end{array}
\right.
\end{equation}
and 
\begin{equation}
G(r,\mu):= \left[1-\left(\frac{r}{R}\right)^2(1-\mu^2)\right]^{1/2}.
\end{equation}
This solution is easily obtained from the formal solution of the
transport equation; see e.g. \cite{Duderstadt79}.

The first three angular moments of the analytical exact solution are
given by
\begin{equation}\label{eq:JHom}
J(r) = \left\{
\begin{array}{lr}
\displaystyle B\left[1-\int_0^1d\mu \cosh(\kappa r \mu){\rm e}^{-\kappa R
    G(r,\mu)}\right]&r<R,\\\\
\displaystyle \frac{B}{2}\left[1-\sqrt{1-\left(\frac{R}{r}\right)^2}-\int_{\sqrt{1-(\frac{R}{r})^2}}^1
  d\mu {\rm e}^{-2\kappa R
    G(r,\mu)}\right]&r\geq R,
\end{array}
\right.
\end{equation}
and
\begin{equation}\label{eq:HHom}
H(r) = \left\{
\begin{array}{lr}
\displaystyle B\int_0^1d\mu \mu\sinh(\kappa r \mu){\rm e}^{-\kappa R
    G(r,\mu)}&r<R,\\\\
\displaystyle \frac{B}{2}\left[\frac{1}{2}\left(\frac{R}{r}\right)^2-\int_{\sqrt{1-(\frac{R}{r})^2}}^1
  d\mu\, \mu{\rm e}^{-2\kappa R
    G(r,\mu)}\right]&r\geq R,
\end{array}
\right.
\end{equation}
and 
\begin{equation}\label{eq:KHom}
K(r) = \left\{
\begin{array}{lr}
\displaystyle B\left[\frac{1}{3}-\int_0^1d\mu \mu^2 \cosh(\kappa r \mu){\rm e}^{-\kappa R
    G(r,\mu)}\right]&r<R,\\\\
\displaystyle \frac{B}{6}\left[1-\left(1-\left(\frac{R}{r}\right)^2\right)^{3/2}\!\!\!\!\!-3\int_{\sqrt{1-(\frac{R}{r})^2}}^1
  d\mu \mu^2 {\rm e}^{-2\kappa R
    G(r,\mu)}\right]\!\!&r\geq R,
\end{array}
\right.
\end{equation}
and can be computed by direct integration from Equation
\eqref{eq:AnaSolHS} using the parity of the function $G(r, \mu)$ with
respect to $\mu$.

We now give some particular values of the analytical solutions that
will be useful for the analysis of the IDSA.
\begin{proposition}\label{prop:JH} The values of $J$ and $H$ at $r=0$
  and at $r=R$ are
\begin{eqnarray}
J(0) &=& B(1-{\rm e}^{-\kappa R}),\label{eq:J0HR}\\
J(R)&=& \frac{B}{2}\left[1 + \frac{{\rm e}^{-2\kappa R}-1}{2\kappa
    R}\right],\\
H(0) &=&0,\\
H(R) &=& \frac{B}{2}\left[\frac{1}{2}+ {\rm e}^{-2\kappa
    R}\left(\frac{1}{2\kappa R}+\frac{1}{(2\kappa
      R)^2}\right)-\frac{1}{(2\kappa R)^2}\right].
\end{eqnarray}
\end{proposition}

\begin{proof} The proof is done by direct integration.
\end{proof}

Before applying the IDSA to the homogeneous sphere, we discuss the limit of the
homogeneous sphere when the opacity $\kappa \to \infty$. We call this
limit the \emph{infinitely opaque homogeneous sphere}.

For the stationary distribution function \eqref{eq:AnaSolHS}, we have
the following limit 
\begin{equation}
\lim_{\kappa\to\infty} f(r,\mu) = \left\{ 
\begin{array}{ll}
0,&\displaystyle \text{if }r\geq R, \ -1<\mu\leq \left[1-\left(\frac{R}{r}\right)^2\right]^{1/2},\\
B,& \text{otherwise.}
\end{array}
\right.
\end{equation}
For the first three angular moments, we have 
\begin{equation}\label{eq:JJ}
\lim_{\kappa\to\infty}J(r) = \left\{
\begin{array}{lr}
\displaystyle B&r<R,\\
\displaystyle \frac{B}{2}\left[1-\sqrt{1-\left(\frac{R}{r}\right)^2}\right]&r\geq R,
\end{array}
\right.
\end{equation}
and
\begin{equation}\label{eq:HH}
\lim_{\kappa\to\infty}H(r) = \left\{
\begin{array}{lr}
0&r<R,\\
\displaystyle \frac{B}{2}\left[\frac{1}{2}\left(\frac{R}{r}\right)^2\right]&r\geq R,
\end{array}
\right.
\end{equation}
and 
\begin{equation}\label{eq:KK}
\lim_{\kappa\to\infty} K(r) = \left\{
\begin{array}{lr}
\displaystyle \frac{B}{3}&r<R,\\
\displaystyle
\frac{B}{6}\left[1-\left(1-\left(\frac{R}{r}\right)^2\right)^{3/2}\right]&r\geq
R. 
\end{array}
\right.
\end{equation}
These results directly come from Eqs
\eqref{eq:JHom}--\eqref{eq:KHom} by noting that the dependence on
$\kappa$ is contained only in the integral terms that vanish when $\kappa \to
\infty$ as shown in the following proposition.
\begin{proposition}
Integrals of the form 
\begin{eqnarray}
I_1 &=& \int_0^1d\mu \mu^{\{0,2\}} \cosh(\kappa r \mu){\rm e}^{-\kappa R
    G(r,\mu)},\\
I_2 &=&\int_0^1d\mu \mu \sinh(\kappa r \mu){\rm e}^{-\kappa R
    G(r,\mu)},\\
I_3 &=& \int_{\sqrt{1-(\frac{R}{r})^2}}^1d\mu\, \mu^{\{0,1,2\}}{\rm e}^{-2\kappa R
    G(r,\mu)},
\end{eqnarray}
vanish when $\kappa \to \infty$.
\end{proposition}
\begin{proof}
The proof for $I_3$ is clear as the only dependence on $\kappa$ is in
the exponential. The proof for $I_1$ and $I_2$ relies on the fact that
the integrants satisfy
\begin{eqnarray}
\mu^{\{0,2\}} \cosh(\kappa r \mu){\rm e}^{-\kappa R
    G(r,\mu)}&=&\mathcal O({\rm e}^{\kappa (r\mu- R
    G(r,\mu))}),\\
\mu \sinh(\kappa r \mu){\rm e}^{-\kappa R
    G(r,\mu)}&=&\mathcal O({\rm e}^{\kappa( r\mu- R
    G(r,\mu))}),
\end{eqnarray}
because the behavior of $\cosh$ and $\sinh$ is dominated by their
positive exponential parts.
For $I_1$ and $I_2$ to vanish in the limit $\kappa \to \infty$, it is sufficient to check that $r\mu- R
    G(r,\mu)<0$ for $r<R$, or equivalently, because all these
    quantities are positive, that $(r\mu)^2- (R
    G(r,\mu))^2<0$ for $r<R$. We obtain
\begin{equation}
\begin{aligned}
(r\mu)^2- (R
    G(r,\mu))^2 &= (r\mu)^2- R^2+r^2(1-\mu^2)\\
&=- R^2+r^2<0,\quad \text{since }r<R.
\end{aligned}
\end{equation}
This concludes the proof.
\end{proof}

In this limit, we can compute the variable Eddington factor $k =
\frac{K}{J}$ and the flux factor $h = \frac{H}{J}$ defined in
Definition~\ref{def:FF}.

\begin{proposition}\label{prop:FluxFactorHS} In the limit $\kappa \to \infty$ the flux factor $h_R$
  and the variable Eddington factor $k_R$ of the homogeneous sphere of
  radius $R$ are given by
\begin{equation}\label{eq:HR}
h_R(r) = \frac{H(r)}{J(r)} = 
\left\{
\begin{array}{ll}
0,&r<R,\\
\displaystyle\frac{1}{2}\left(1+\sqrt{1-\left(\frac{R}{r}\right)^2}\right),&r\geq R,
\end{array}
\right.
\end{equation}
and
\begin{equation}\label{eq:KR}
k_R(r) = \frac{K(r)}{J(r)} =
\left\{
\begin{array}{ll}
\displaystyle\frac{1}{3},&r<R,\\
\displaystyle\frac{1}{3}\left(2-\left(\frac{R}{r}\right)^2+\sqrt{1-\left(\frac{R}{r}\right)^2}\right),&r\geq R.
\end{array}
\right.
\end{equation}
\end{proposition}
\begin{proof} The proof is done by direct computation from equations \eqref{eq:JJ}-\eqref{eq:KK}.
\end{proof}

With these results on the behavior of the analytical solution of the
homogeneous sphere, we can now study the IDSA in this
particular case.

\subsection{Applying the IDSA to the homogeneous
 sphere} 
In order to apply the IDSA to the homogeneous sphere test case and to overcome the numerical problems illustrated in Subsection \ref{sub:issues}, we need to reformulate the IDSA. This is done by introducing the following assumption. 
\begin{assumption}\label{ass:RisRnu}
We assume that the radius $R$ of the
homogeneous sphere cor\-res\-ponds to the neutrinosphere radius. We also
assume that
\begin{enumerate}
\item The domain of validity of the diffusion and the reaction regimes
  is cha\-rac\-te\-ri\-zed by $r<R$.
\item The domain of validity of the free streaming regime is
  characterized by $r\geq R$.
\end{enumerate}
\end{assumption}
This is a very important assumption because it allows us to avoid the
problematic diffusion source $\Sigma$ and to realize the coupling
directly and explicitly. This way, we eliminate both the spurious
trapped and the instability problem discussed above and we will have
access to the modeling error of the IDSA.
\begin{remark}\label{rem:gdemi} Assumption \ref{ass:RisRnu} allows us to simplify the
  Eqs \eqref{eq:str.reac} and \eqref{eq:str.reac.2} for the streaming
  particles and \eqref{eq:str.diff} and \eqref{eq:str.diff1} for the
  trapped particles, recalling \eqref{eq:J_sr.free}, because the
  function  $h^{\rm free} = g_{\rm idsa} = \frac{1}{2}$ in these cases.
\end{remark}
\begin{remark}\label{rem:Rnu}We remark here that Assumption
  \ref{ass:RisRnu} does not 
  match the usual definition of the neutrinosphere based on the optical
  depth, see Eq. \eqref{eq:tau}. In fact, in the case of the
  homogeneous sphere, one can compute the neutrinosphere radius
  $R_\nu$
\begin{equation}
\int_{R_{\nu}}^\infty \kappa_{\rm a}(r) dr = \kappa\int_{R_{\nu}}^R dr=\kappa(R-R_{\nu}) =\frac{2}{3},
\end{equation}
where we used the definition of $\kappa_{\rm a}$, see
Eq. \eqref{eq:kappaHS}; that is, the usual definition of the
neutrinosphere radius only corresponds to the homogeneous sphere
radius in the limit $\kappa \to \infty$ as $R-R_{\nu} = \frac{2}{3\kappa}$.
\end{remark}
\begin{remark}[IDSA in the limit $\kappa \to \infty$]
In the limit $\kappa \to \infty$, the diffusion limit is irrelevant
and we have the following solution
\begin{eqnarray}
J^t_{\rm reac}(r,t) &=& B,\quad r< R,\\
J^s_{\rm reac}(r,t) &=& 0,\quad r< R,\\
J^t_{\rm free}(r,t) &=& 0,\quad r \geq R,\\
J^s_{\rm free}(r,t) &=&
C(t)\frac{2}{R^2}\left(1-\sqrt{1-\left(\frac{R}{r}\right)^2}\right),\quad
r \geq R.
\end{eqnarray}
If we set $C(t):= \frac{R^2B}{4}$, the solution of this coupling is
equivalent to the analytical solution \eqref{eq:JJ} of the infinitely
opaque homogeneous 
sphere. This condition can be ensured by requiring the boundary
condition $J^s_{\rm free}(R,t) = \frac{B}{2}$.
\end{remark}
Using Assumption \ref{ass:RisRnu} and Remark \ref{rem:gdemi}, one can rewrite the IDSA equations \eqref{eq:str.diff1}--\eqref{eq:str.free} without reaction and with $\kappa =0$ for $r\geq R$ as 
\begin{equation}\label{eq:IDSAsansReac}
\begin{array}{rclr}
\displaystyle\partial_t J^t -\frac{1}{r^2}\partial_r\left(\frac{r^2}{3\kappa}\partial_r
      J^t\right)-\frac{2}{3 }\partial_rJ^t &=& \kappa(B-J^t),&r<R,\\
J^s &=&\displaystyle -\frac{2}{3\kappa}\partial_r J^t, & r<R,\\
J^t &=& 0,& r\geq R,\\
\displaystyle\frac{1}{r^2}\partial_r\left(r^2g(r) J^s\right) &=&
0,&r\geq R,
\end{array}
\end{equation}
with $g(r)=
\frac{1}{2}\left(1+\sqrt{1-\left(\frac{R}{r}\right)^2}\right)$ the
geometrical factor appearing in Eq. \eqref{eq:hfree}.

It is possible to reconstruct the first and second angular moments $H$
and $K$ by defining flux factors for the trapped and the free
streaming components. We set
\begin{equation}\label{eq:FFtrapped}
h^t_R := 0,\qquad h^s_R:=\left\{
\begin{array}{ll}
\displaystyle\frac{1}{2},&r<R,\\
\displaystyle\frac{1}{2}\left(1+\sqrt{1-\left(\frac{R}{r}\right)^2}\right),&r\geq R,
\end{array}
\right.
\end{equation}
and
\begin{equation}\label{eq:FFfree}
k^t_R := \frac{1}{3},\qquad k^s_R:=\left\{
\begin{array}{ll}
\displaystyle\frac{1}{3},&r<R,\\
\displaystyle\frac{1}{3}\left(2-\left(\frac{R}{r}\right)^2+\sqrt{1-\left(\frac{R}{r}\right)^2}\right),&r\geq R.
\end{array}
\right.
\end{equation}

\begin{remark} We note that the definition of $h^s_R$ corresponds to
  the definition of $h^{\rm free} = g_{\rm idsa}$, see Eq. \eqref{eq:hfree}. In
  fact, this definition can be 
  used as a justification for the definition of $g_{\rm idsa}$ in the
  derivation of the IDSA. The definition of these flux factors
  extend the stationary free streaming limit to the first
  order moment solution.
\end{remark}

\begin{proposition}\label{prop:FFIDSA} With Equations \eqref{eq:FFtrapped} and
  \eqref{eq:FFfree}, we can reconstruct $H_{\rm idsa}(r)$ and $K_{\rm
    idsa}(r)$ through the relations
\begin{eqnarray}
H_{\rm idsa}(r) &=&h^t_R(r)J^t(r) +h^s_R(r)J^s(r) =h^s_R(r)J^s(r),\label{eq:HIDSA} \\
K_{\rm idsa}(r) &=&k^t_R(r)J^t(r) +k^s_R(r)J^s(r).
\end{eqnarray}
This reconstruction has the following properties:
\begin{enumerate}
\item IDSA closure in the diffusion regime:
$$H_{\rm idsa} = -\frac{1}{3\kappa}\partial_r J^t,\  \text{if }r<R.$$
\item Free streaming flux ratio:
$$h^s_R(r)=h^{\rm idsa}_R(r) = h_R(r),\  \text{if }r\geq R.$$
\item Diffusion variable Eddington factor:
$$k^{\rm idsa}_R(r) = \frac{1}{3} = k_R(r),\  \text{if }r<R.$$
\item Free streaming variable Eddington factor:
$$k^s_R(r)=k^{\rm idsa}_R(r) = k_R(r),\  \text{if }r\geq R,$$
\end{enumerate}
where $h_R$ and $k_R$ are the flux factors for the infinitely opaque
homogeneous sphere 
and $h^{\rm idsa}_R:= \frac{H_{\rm idsa}}{J^t+J^s}$ and $k^{\rm
  idsa}_R:= \frac{K_{\rm idsa}}{J^t+J^s}$ are the IDSA flux factors.
\end{proposition}

\begin{proof}
We prove separatly the four points.
\begin{enumerate}
\item The first statement can be proved by inserting Eqs
  \eqref{eq:FFtrapped} and \eqref{eq:IDSAsansReac} into Eq.~\eqref{eq:HIDSA}.
\item In the region $r\geq R$, we have $J^t = 0$. As a consequence,
  $h_R^s = h^{\rm idsa}_R = h_R$ by Eqs \eqref{eq:FFtrapped} and
  \eqref{eq:HR}. 
\item In the region $r<R$, we have $k^t_R = k^s_R = 1/3$. Therefore, 
$$
K_{\rm idsa}(r) =k^t_R(r)J^t(r) +k^s_R(r)J^s(r) =
\frac{1}{3}(J^t(r)+J^s(r)) = \frac{1}{3}J_{\rm idsa}(r).
$$
The proof is completed by comparing this result with Eq. \eqref{eq:KR}.
\item The proof of this point is similar to the proof of 2.
\end{enumerate}
This concludes the proof.
\end{proof}

The proposition \ref{prop:FFIDSA} shows that the flux factors of the
IDSA correspond to the flux factors of the infinitely opaque homogeneous sphere. However, the flux ratio of the IDSA in the
diffusion region $r<R$ does not match the diffusion limit relation
\eqref{eq:diff.closure} even if the form is similar. In fact,
\begin{equation}\label{eq:diffEqHR}
H_{\rm idsa} = -\frac{1}{3\kappa}\partial_r J^t \neq
-\frac{1}{3\kappa}\partial_r J_{\rm idsa} = -\frac{1}{3\kappa}\partial_r J^t +\frac{2}{9\kappa^2}\partial^2_r J^t,
\end{equation}
where we used Eq. \eqref{eq:IDSAsansReac}$_2$ to eliminate the streaming
component in the last equality.
This means that the diffusion limit is not well represented by this
formulation. Moreover, the diffusion equation \eqref{eq:tr.with.adv} used for the trapped
particles does not correspond to the usual diffusion approximation
equation \eqref{eq:J_sr.diff}.

Another possible problem of this formulation is that the amount of
streaming particles at the homogeneous sphere radius is not
necessarily accurate. If this value is inaccurate, then the error
will be propagated by the free streaming equation in the whole domain
where $r\geq R$. This can create large approximation errors.

In order to deal with the problem of the advection term and the
possible inaccuracy of the amount of streaming particles at the
homogeneous sphere radius, we propose in the next subsection some
possible modifications of the approximation.

\subsection{Improvements of the scheme}\label{sec:Impro}
We now propose some improvements of the IDSA. We propose to use the
diffusion equation \eqref{eq:J_sr.diff} instead of
Eq. \eqref{eq:IDSAsansReac}$_1$ for the evolution of the 
trapped particles and to
modify the reconstruction of the streaming 
particles in the diffusion regime to correctly predict the amount of
streaming particles at the homogeneous sphere radius. This is done through the following assumption
\begin{assumption}
In the diffusion region, we assume that the free streaming component
can be reconstructed from the trapped component by a relation of the
form
\begin{equation}
J^s(r) = C \partial_r J^t.
\end{equation}
The constant $C$ is chosen in such a way that $J^s(R) = J(R) = \frac{B}{2}\left[1 + \frac{{\rm e}^{-2\kappa R}-1}{2\kappa
    R}\right]$, see Proposition \ref{prop:JH}.
\end{assumption}
The scheme obtained can be written as
\begin{equation}\label{eq:IDSANew1}
\begin{array}{rclr}
\displaystyle\partial_t J^t -\frac{1}{r^2}\partial_r\left(\frac{r^2}{3\kappa}\partial_r
      J^t\right)\!\!\!\!&=&\!\!\!\! \kappa(B-J^t),&r<R,\\
J^s \!\!\!\!&=&\!\!\!\!\displaystyle C\partial_rJ^t = \frac{B}{2}\left[1 + \frac{{\rm e}^{-2\kappa R}\!-1}{2\kappa
    R}\right]\frac{\partial_r J^t}{\partial_r J^t|_{r=R}},\!\!\! & r<R,\\
J^t \!\!\!\!&=&\!\!\!\! 0,& r\geq R,\\
\displaystyle\frac{1}{r^2}\partial_r\left(r^2g(r) J^s\right) \!\!\!\!&=&\!\!\!\!
0,&r\geq R,
\end{array}
\end{equation}
where $g(r) =
\frac{1}{2}\left(1+\sqrt{1-\left(\frac{R}{r}\right)^2}\right)$. The
results 2.--4. of Proposition \ref{prop:FFIDSA} still hold and the
point 1. can be rephrased as 
\begin{itemize}
\item[1'.] $H_{\rm idsa} = \frac{C}{2}\partial_r J^t$, if $r<R$.
\end{itemize}
The proof of this fact closely follows the proof of point 1. of
Proposition~\ref{prop:FFIDSA}.

The scheme ensures that the trapped particles are evolved by a
diffusion equation and that the amount of streaming particles is
correct in the whole domain where $r\geq R$. Another advantage of this
formulation is that it has a stationary state closed form solution.

In the stationary state limit, we can compute the closed form solution
of system~\eqref{eq:IDSANew1} given in the following proposition. 
\begin{proposition}\label{prop:IDSANew}
The closed form solution of the stationary state limit of
\eqref{eq:IDSANew1} is given by 
\begin{equation}\label{eq:IDSANewSol}
\begin{array}{rclr}
J^t(r) &=&\displaystyle B \left(1-\frac{R}{r}\frac{\sinh(\sqrt{3}\kappa r)}{\sinh(\sqrt{3}\kappa R)}\right),&r<R,\\
J^s(r) &=&\displaystyle C \frac{3\kappa^2BR}{\sinh(\sqrt{3}\kappa
  R)}\left(\frac{\sinh(\sqrt{3}\kappa r)}{(\sqrt{3}\kappa r)^2}-\frac{
    \cosh(\sqrt{3}\kappa r)}{\sqrt{3}\kappa r}\right), & r<R,\\
J^t(r) &=& 0,& r\geq R,\\
J^s(r) &=&\displaystyle\frac{B}{2}\left[1 + \frac{{\rm e}^{-2\kappa R}-1}{2\kappa
    R}\right]\left(1-\sqrt{1-\left(\frac{R}{r}\right)^2}\right),&r\geq R,
\end{array}
\end{equation}
where $C = \frac{R}{2}\left[1 + \frac{{\rm e}^{-2\kappa R}-1}{2\kappa
    R}\right]\left(1-\frac{\sqrt{3}\kappa R}{\tanh(\sqrt{3}\kappa
    R)}\right)^{-1}$. 

\end{proposition}
\begin{proof} To prove this proposition, it is enough to substitute the proposed solution into system \eqref{eq:IDSANew1} and checking that it is a solution.
\end{proof}
\begin{remark}\label{rem:NIDSAprop} The trapped component of the solution \eqref{eq:IDSANewSol} has the property to have a negative radial derivative that vanishes at $r=0$, see for example \cite[Chap. 4]{MichaudPHD}. This ensures that the streaming component is positive. The vanishing at $r=0$ property motivates the use of a Neumann boundary condition at this point.
\end{remark}

Proposition \ref{prop:IDSANew} gives a closed form solution of the
new version of the IDSA. 
It can be shown 
that the total solution $J_{\rm idsa} = J^t + J^s$ is bounded by
$B$. As a consequence of Remark \ref{rem:NIDSAprop}, we have $J^s(0)=0$ and $J_{\rm
  idsa}(0) = J^t(0) = B\left(1-\frac{\sqrt{3}\kappa R}{\sinh(\sqrt{3}\kappa
      R)}\right)$, since $\lim_{x\to 0} \frac{\sinh(x)}{x} = 1$. This value has to be compared with the value of the
  solution of the homogeneous sphere given in Proposition~\ref{prop:JH}.
 
   This comparison leads to the following result.
  \begin{figure}[t]
\centering
\includegraphics[width = .48\linewidth]{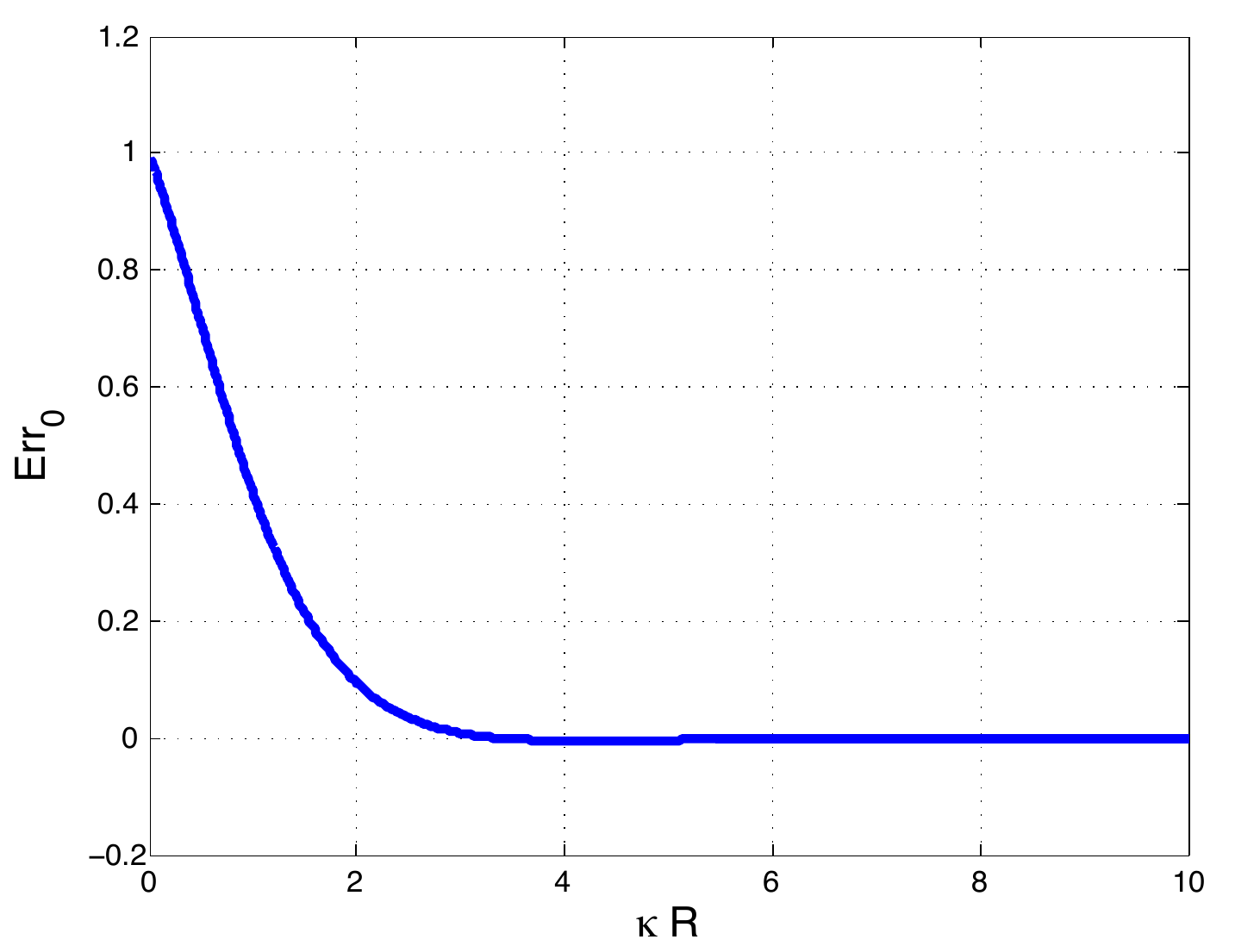}
\caption{Plot of the value of ${\rm Err}_0$ as a function of $\kappa R$. The
  plot shows that if $\kappa R$ is large enough, the 
  error in $r=0$ vanishes.}
\label{fig:Err0}
\end{figure}
\begin{proposition}
The relative error at $r=0$ of the New IDSA approximation~\eqref{eq:IDSANewSol} is given by
\begin{equation}
{\rm Err}_0 = \frac{\left(\frac{\sqrt{3}\kappa R}{\sinh(\sqrt{3}\kappa
      R)}-{\rm e}^{-\kappa R}\right)}{(1-{\rm e}^{-\kappa R})}.
\end{equation}
\end{proposition}
\begin{proof} It suffices to substract the evaluation of \eqref{eq:IDSANewSol}$_1$ at $r=0$, using the fact that $\lim_{x\to 0} \frac{\sinh(x)}{x} = 1$, from \eqref{eq:J0HR} and divide by
  Eq. \eqref{eq:J0HR}. 
\end{proof}
The value of ${\rm Err}_0$ can be used as a measure of the
validity of the 
New IDSA approximation~\eqref{eq:IDSANew1}.

Figure \ref{fig:Err0} shows that the error vanishes when $\kappa R$ is
large. This figure also shows that the approximation fails when $\kappa R$ is
small, that is when the diffusion approximation is not valid, as
expected.

\section{Numerical experiments}\label{sec:Num}
For the numerical experiments, we apply the two versions of the IDSA
designed before to the homogeneous sphere test case. We refer to
Eq. \eqref{eq:IDSAsansReac} as the \emph{Old IDSA} approximation
and to Eq. \eqref{eq:IDSANew1} as the \emph{New IDSA}
approximation. We use a backward Euler scheme with a finite
difference scheme in space with a very fine grid, $N_r =
10^6$. 
Such a fine grid eliminates discretization errors and the
remaining discrepancy can be attributed to modeling error. We use $0$ initial conditions and boundary conditions are discussed in \cite[Chap.4]{MichaudPHD}.\\

The test case we choose is an homogeneous sphere of radius $R=6$. The
value of $\kappa$ is varied in the different tests.
\begin{figure}[t]
\centering
\includegraphics[width = .48\linewidth]{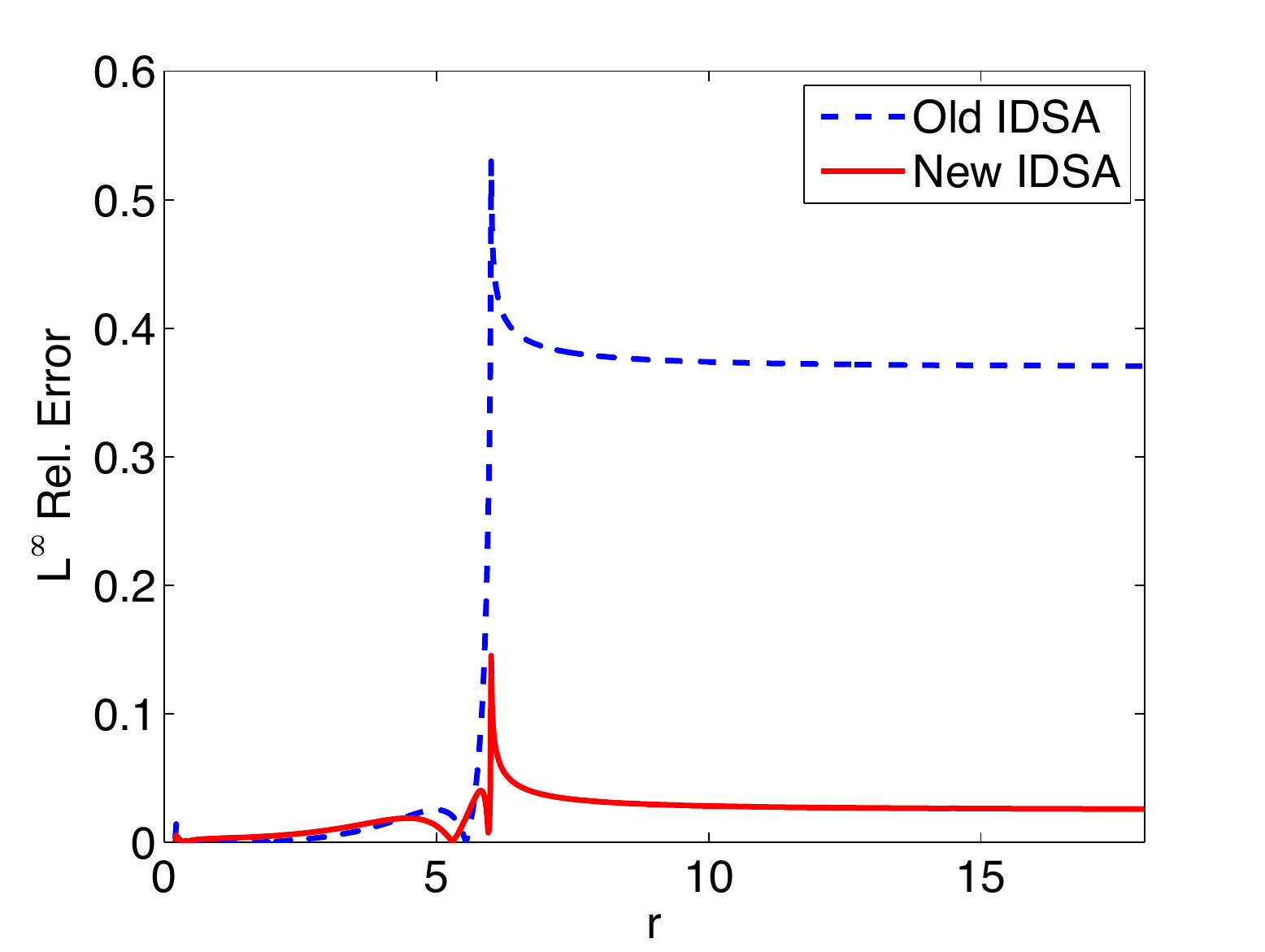}
\caption{Plot of the pointwise relative error of the Old IDSA (in blue) and of
  the New IDSA (in red) applied to the homogeneous sphere of radius
  $R=6$ and $\kappa = 1$.}
\label{fig:ErrIDSAoldnew}
\end{figure}

As a first example, we apply the Old IDSA and the New IDSA to the
homogeneous sphere of radius $R=6$ with $\kappa = 1$. Figure
\ref{fig:ErrIDSAoldnew} shows the relative approximation
error of these two schemes. We notice that the Old IDSA (in blue)
has a considerably larger error in the free streaming region where
$r\geq R$ than the New IDSA (in red). These results show that even if
the flux factors are accurate and the 
diffusion relation given in Proposition \ref{prop:FFIDSA} is
satisfied by the Old IDSA approximation, it does not mean that the
approximation is good. The mo\-di\-fi\-ca\-tions given is Section
\ref{sec:Impro} greatly improve the situation. The error still has a
peak at the homogeneous sphere radius, but this is expected as a
consequence of the definition of approximations
\eqref{eq:IDSAsansReac} and \eqref{eq:IDSANew1}. A better view of this
example is given in Figure \ref{fig:OldAndNewIDSA} where we also give 
the value of the flux factors $h$ and $k$ which can be computed from
the value of $J^t$ and $J^s$ from Proposition \ref{prop:FFIDSA}. 

\begin{figure}[t]
\centering

\subfigure{\label{Fig:I-4-OJ.a}\includegraphics[width=.482\textwidth]{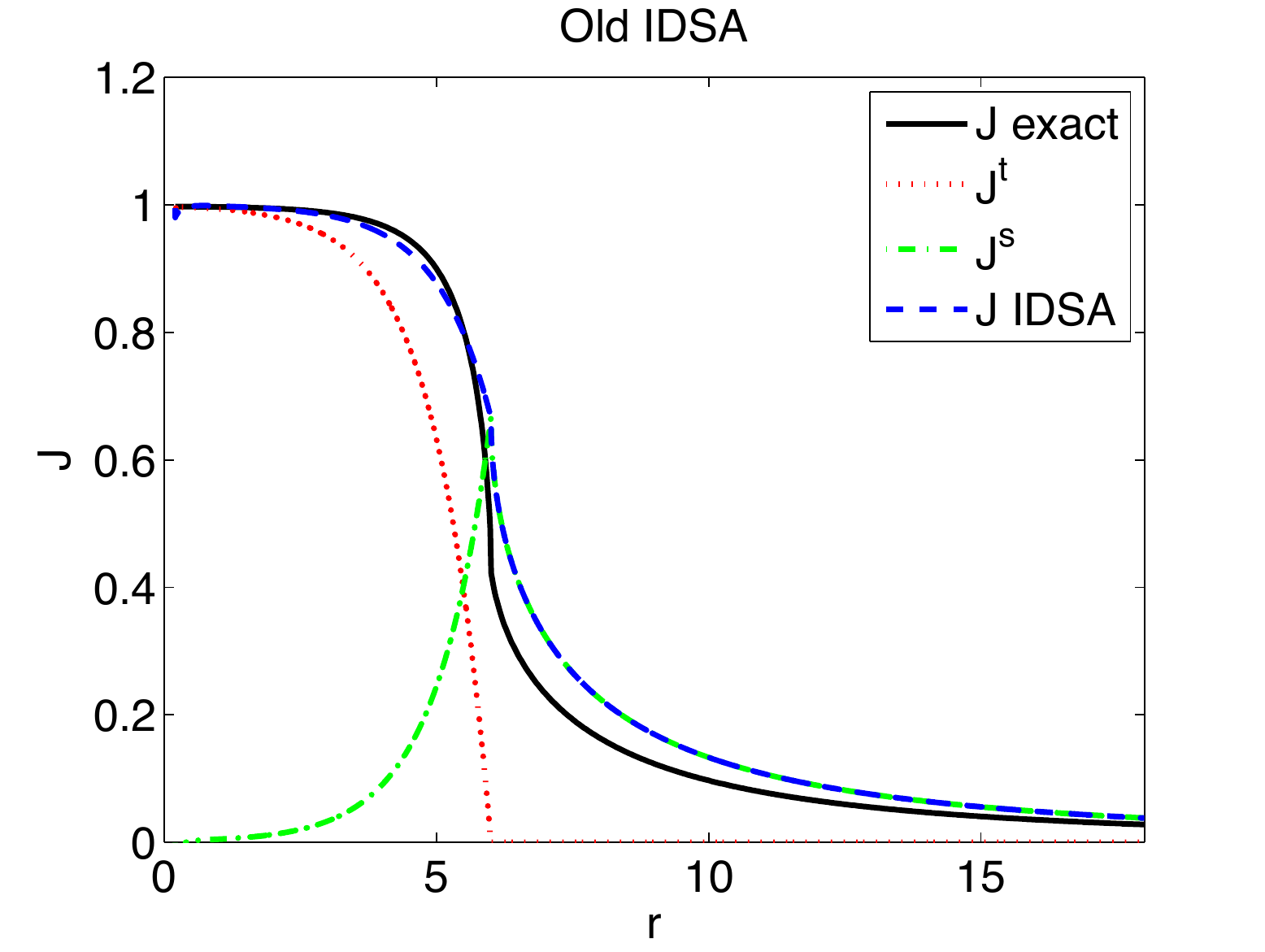}}
\
\subfigure{\label{Fig:I-4-NJ.b}\includegraphics[width=.482\textwidth]{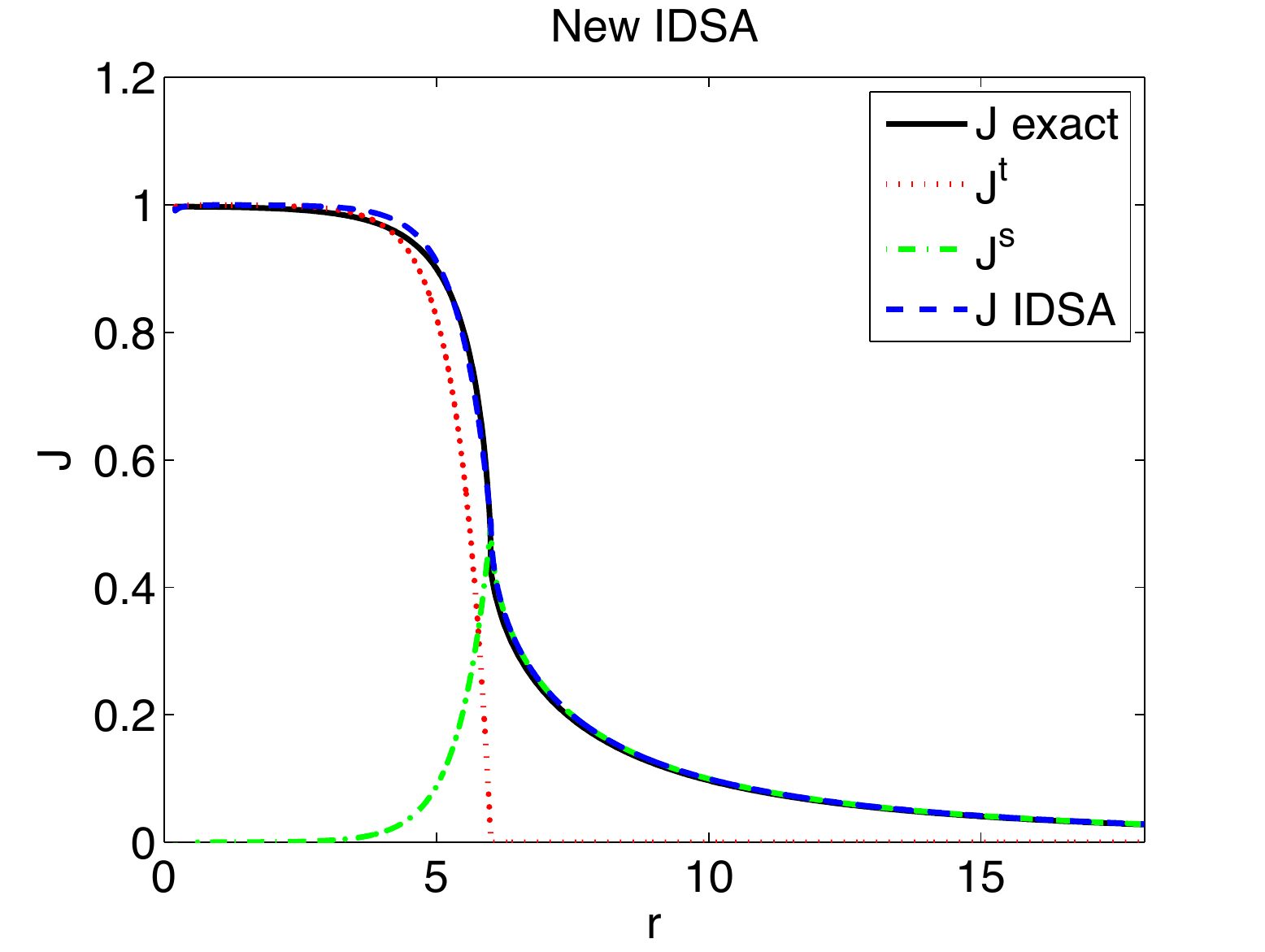}}
\hfill

\subfigure{\label{Fig:I-4-Oh.c}\includegraphics[width=.482\textwidth]{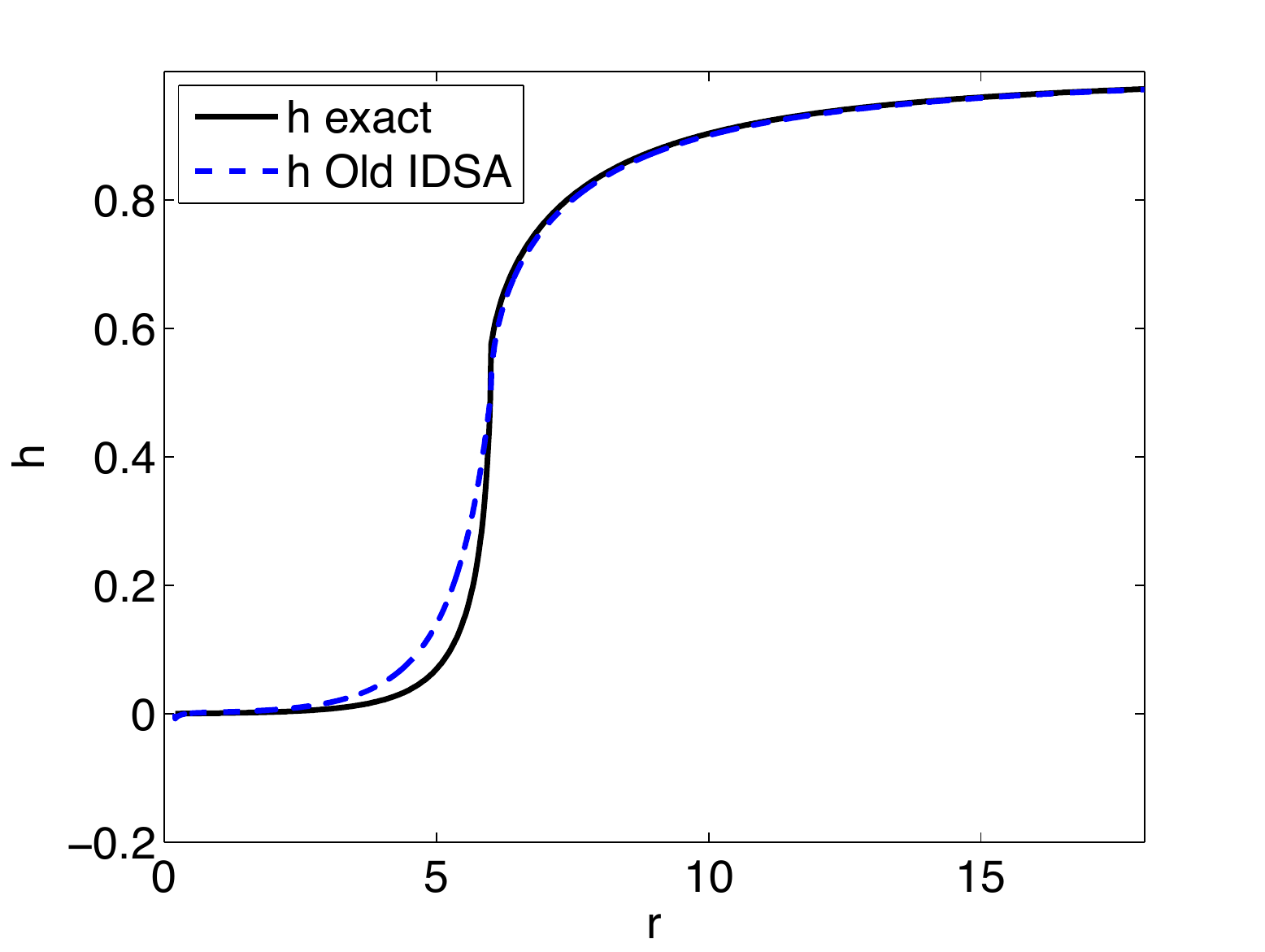}}
\
\subfigure{\label{Fig:I-4-Nh.d}\includegraphics[width=.482\textwidth]{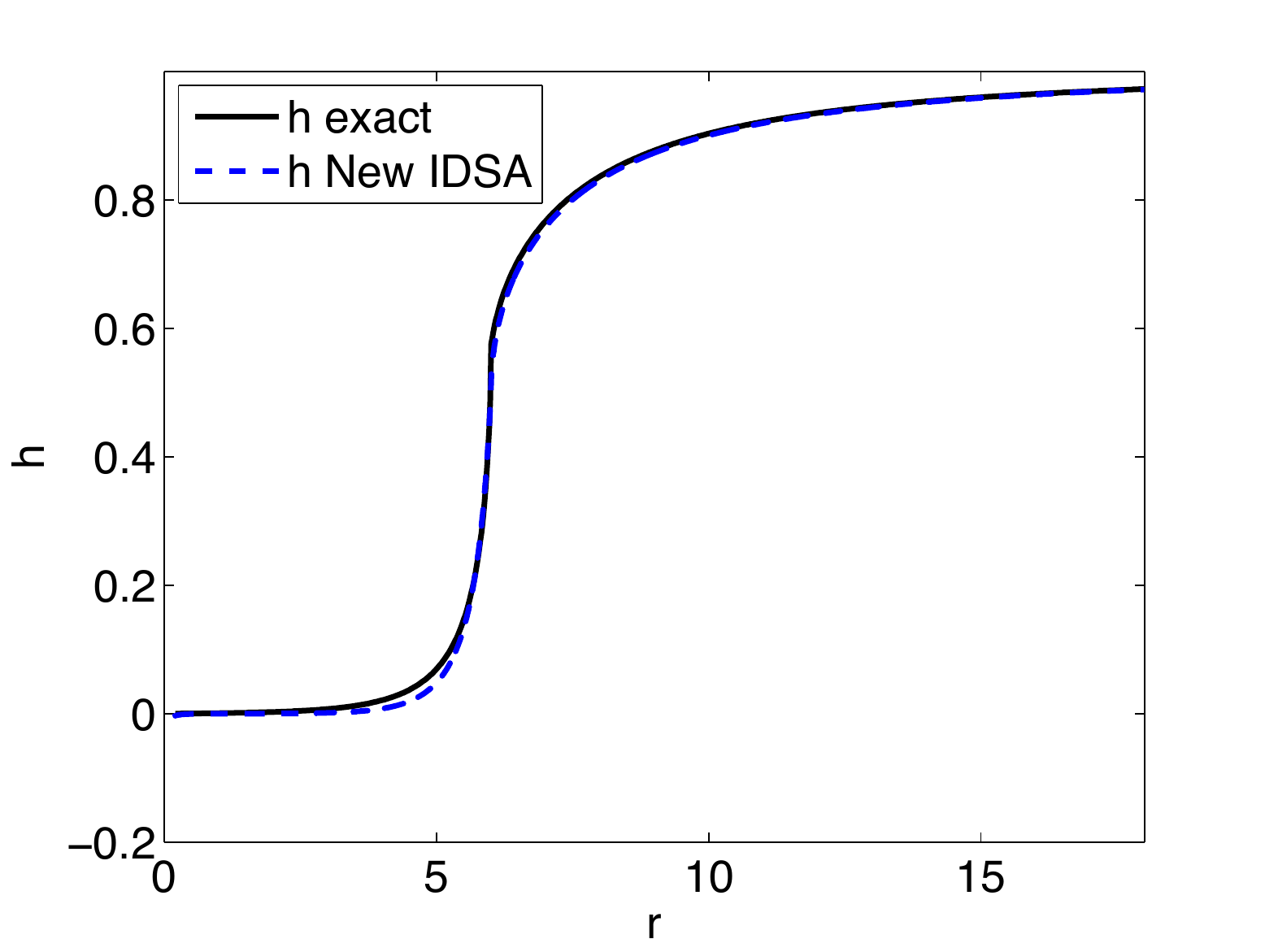}}
\hfill

\subfigure{\label{Fig:I-4-Ok.e}\includegraphics[width=.482\textwidth]{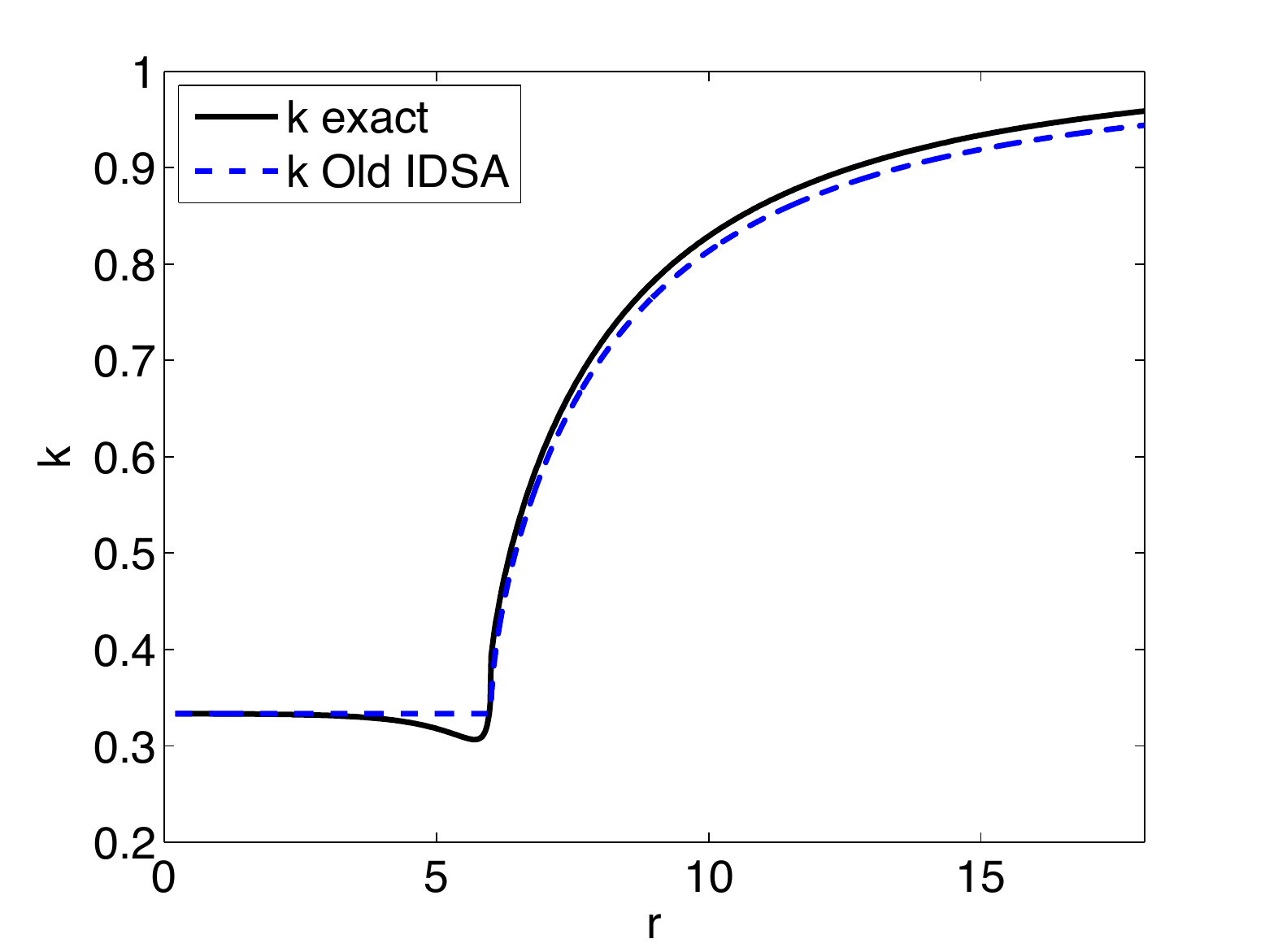}}
\
\subfigure{\label{Fig:I-4-Nk.f}\includegraphics[width=.482\textwidth]{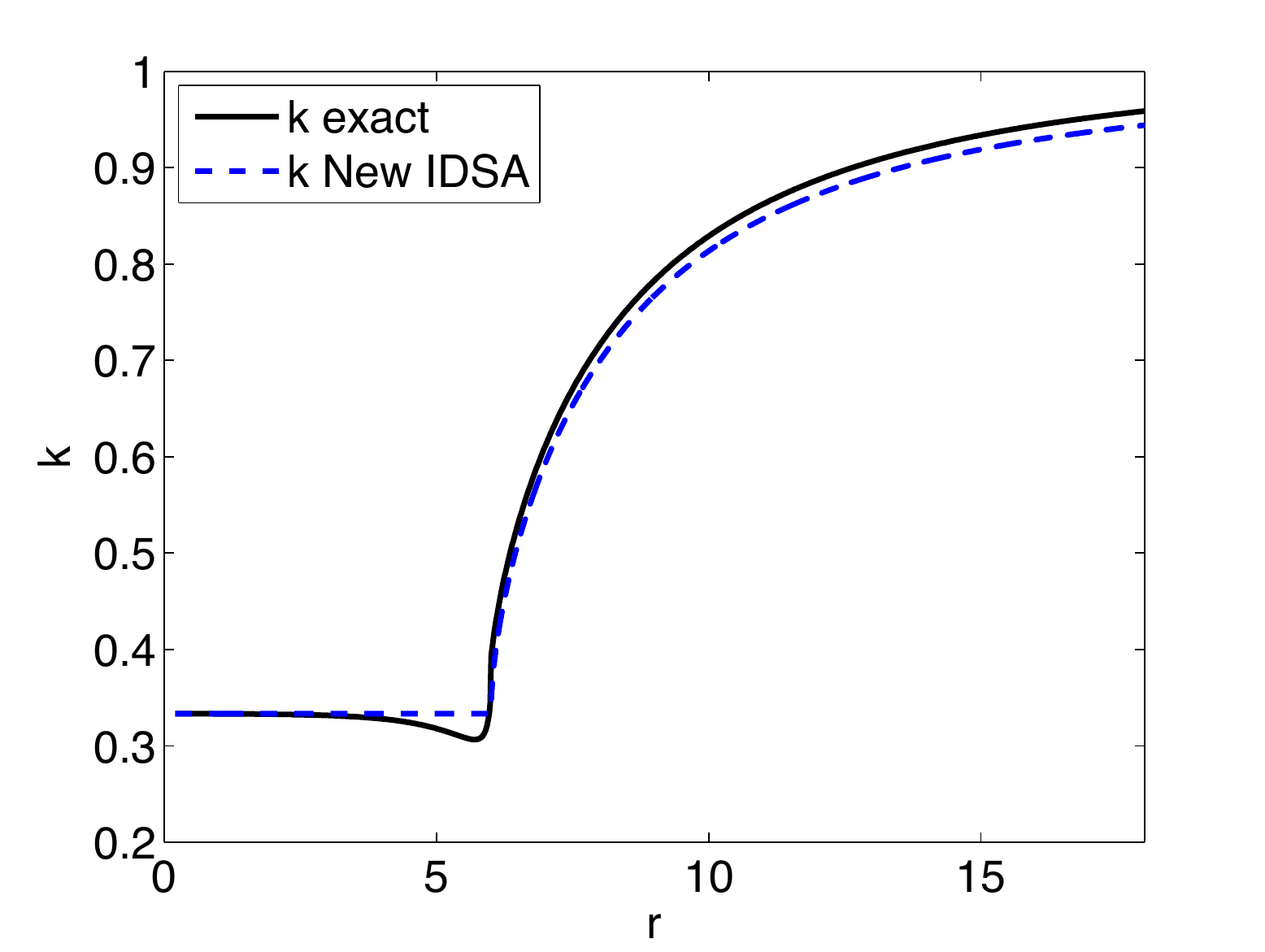}}
\hfill

\caption{Results of the Old IDSA (on the left) and of the New IDSA
  (right) on the homogeneous sphere of radius $R=6$ with $\kappa
  =1$. The first line gives the value of $J$, the second gives the
  flux ratio $h$ and the third gives the variable Eddington factor
  $k$. The exact solution is given by the black line and numerical
  results are given in blue. For the results of $J$, the trapped
  component is given in red and the free streaming component in green. }
\label{fig:OldAndNewIDSA}
\end{figure}
The results show that the Old IDSA overestimates the value of
streaming particles at the radius of the homogeneous sphere; this
leads to an overestimation of the streaming particles in the whole
region where $r\geq R$, explaining the results of Figure
\ref{fig:ErrIDSAoldnew}. We also notice that, compared to the New
IDSA, the Old IDSA underestimates the proportion of the trapped
particles and overestimates the proportion of the streaming
particles in the diffusive region ($r<R$). In a CCSN experiment,
the overestimation of streaming particles can lead to larger heating rates behind the shock and an artificially increased likeliness for explosions. From this viewpoint, one would say that the Old IDSA overestimates the neutrino heating, see \cite{LiebendoerferEtAl09big}. 
Regarding
the flux factors, we see in Figure \ref{fig:OldAndNewIDSA} that the
flux ratio of the Old IDSA is worse 
that the flux ratio of the New IDSA; that is the correction done in
the $J$ distribution reflects in the flux ratio. There are, however, no
differences in the approximation of the variable Eddington
factors. This is a consequence of the reconstruction of $H_{\rm idsa}$
and $K_{\rm idsa}$ given in Proposition \ref{prop:FFIDSA} from which
the flux factors are computed.
\begin{figure}[t!]
\centering
\includegraphics[width = \linewidth]{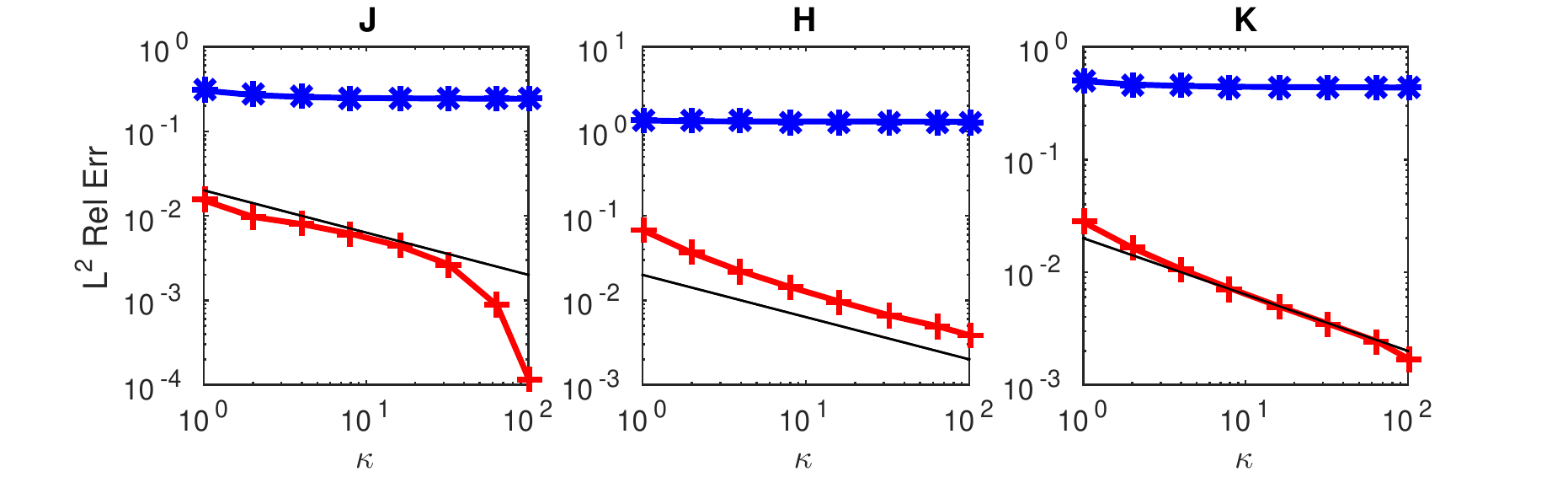}
\caption{Plot of the convergence of the $L^2$ relative errors for the
  homogeneous sphere test case. The radius of the homogeneous sphere
  is $R = 6$ and $\kappa$ is varied from $1$ to $100$. The results
  of the Old IDSA (in blue) do not converge when $\kappa \to \infty$
  whereas the results of the New IDSA (in red) converge to the exact
  solution of the homogeneous sphere. The numerical rate of
  convergence is of the order $\mathcal O(\kappa^{-1/2})$ for all the quantities. }
\label{fig:ConvergenceJHK}
\end{figure}

To complement the study, we let the values of $\kappa$ vary from $1$
to $100$ and compute the relative $L^2$ errors. In Figure
\ref{fig:ConvergenceJHK}, we see that the errors 
of the Old IDSA (in blue) do not converge even if the flux factors are
pretty accurate according to Figure \ref{fig:OldAndNewIDSA}, whereas
the errors of the New IDSA (in red) converge to the exact values of
 $J$, $H$ and $K$ of the homogeneous sphere given in
 \eqref{eq:JHom}--\eqref{eq:KHom}. The numerical order of 
convergence of these three quantities is $\mathcal
O(\kappa^{-1/2})$. This rate of convergence has only been obtained
numerically and we want to prove this fact in future work.

The first panel of Figure
\ref{fig:ConvergenceJHK} shows that the convergence of $J$ accelerates
for large values of $\kappa$, but this behavior has yet to be understood.
\section{Conclusion and outlook}\label{sec:Conclusion}
In this paper, we have discussed the behavior of the IDSA on the
homogeneous sphere test case. We have shown that the current
implementation of this approximation has modeling issues coming from
the de\-fi\-ni\-tion of the coupling mechanism, that is from the diffusion
source $\Sigma$. The two main problems are the \emph{Spurious Trapped}
and \emph{Instability} problems. In the case of the homogeneous
sphere, we have changed the coupling mechanism assuming the different
regions where each regime apply to be known. Doing so, the two
modeling problems are avoided allowing us to access the modeling error
of the IDSA. We show with numerical experiments that the IDSA is
overestimating the streaming component both in the diffusive and in
the free streaming part of the domain. To address the other issue, we
proposed a new version of the IDSA based on the analytical solution of
the homogeneous sphere. This very particular model has been shown to
converge to the solution of the infinitely opaque homogeneous sphere case. We also used the analytical solution of this case to obtain the flux ratio and the
variable Eddington factor for the
stationary state free streaming limit used in the IDSA. The new
version of the IDSA presented has also been shown to converge for the
first and second order moments with these flux factors. The order of
convergence has been numerically estimated to $\mathcal
O(\kappa^{-1/2})$ and the proof of this fact is still missing.

We have shown in this paper that simplifying the model can shed some
light on the behavior of the approximation. To go further in the
analysis, we need to go back to more complicated radiative transfer
problems. The analysis presented in this work is very specific to the
homogeneous sphere example. However, many astrophysical problems do
have a very similar structure, a dense body surrounded by a
transparent atmosphere. This is the case for example in CCSN
simulations, stellar winds, planetary disks and many other
astrophysical problems. A possible direction in which our results can
be extended is to use the neutrinosphere radius as the limiting
radius for trapped particles. Note that Remark \ref{rem:Rnu} shows
that these two radius are usually not equal, but close when $\kappa$
is large. This allows the computation of trapped
particles with the New IDSA scheme. Then, one needs to find a way to
reconstruct the streaming particle distribution inside the
neutrinosphere in order to obtain a boundary condition for the
transport of streaming neutrinos outside the neutrinosphere. Once such
a condition is obtained, one can use any radiative transfer code to
compute the outside distribution.

\bibliographystyle{AIMS}
\bibliography{DCDSSJM}

\providecommand{\href}[2]{#2}
\providecommand{\arxiv}[1]{\href{http://arxiv.org/abs/#1}{arXiv:#1}}
\providecommand{\url}[1]{\texttt{#1}}
\providecommand{\urlprefix}{URL }
\begin{thebibliography}{1}

\bibitem{BeFrGaLiMiVa12_esaimbig}
\newblock H.~Berninger, E.~Fr\'{e}nod, M.~Gander, M.~Liebend\"{o}rfer,
  J.~Michaud and N.~Vasset,
\newblock A {M}athematical {D}escription of the {IDSA} for {S}upernova
  {N}eutrino {T}ransport, its {D}iscretization and a {C}omparison with a
  {F}inite {V}olume {S}cheme for {B}oltzmann's {E}quation,
\newblock in \emph{ESAIM: Proc., CEMRACS'11: Multiscale Coupling of Complex
  Models in Scientific Computing}, vol.~38, 2012,
\newblock 163--182.

\bibitem{BeFrGaLiMiVa12big}
\newblock H.~Berninger, E.~Fr\'{e}nod, M.~Gander, M.~Liebend\"{o}rfer,
  J.~Michaud and N.~Vasset,
\newblock Derivation of the {IDSA} for {S}upernova {N}eutrino {T}ransport by
  {A}symptotic {E}xpansions,
\newblock \emph{SIMA}, \textbf{45} (2013), 3229--3265.

\bibitem{Bruenn85}
\newblock S.~Bruenn,
\newblock Stellar core collapse: {N}umerical model and infall epoch,
\newblock \emph{Astrophys.\ J.\ Suppl.\ S.}, \textbf{58} (1985), 771--841.

\bibitem{Duderstadt79}
\newblock J.~J. Duderstadt and W.~R. Martin,
\newblock \emph{Transport Theory},
\newblock John Wiley \& Sons, 1979.

\bibitem{LevermorePomraning81}
\newblock C.~Levermore and G.~Pomraning,
\newblock A flux-limited diffusion theory,
\newblock \emph{The Astrophysical Journal}, \textbf{248} (1981), 321--334.

\bibitem{LiebendoerferEtAl09big}
\newblock M.~Liebend\"orfer, S.~Whitehouse and T.~Fischer,
\newblock The {I}sotropic {D}iffusion {S}ource {A}pproximation for {S}upernova
  {N}eutrino {T}ransport,
\newblock \emph{Astrophys.\ J.}, \textbf{698} (2009), 1174--1190.

\bibitem{MichaudPHD}
\newblock J.~Michaud,
\newblock \emph{From Neutrino Radiative Transfer in Core-Collapse Supernovae to
  Fuzzy Domain Based Coupling Methods},
\newblock PhD thesis, Universit\'{e} de Gen\`{e}ve, 2015.

\bibitem{Mihalas84}
\newblock D.~Mihalas and B.~Weibel-Mihalas,
\newblock \emph{Foundation of Radiation Hydrodynamics},
\newblock Oxford University Press, 1984.

\bibitem{SmitHornBludman00}
\newblock J.~Smit, L.~Van~den Horn and S.~Bludman,
\newblock Closure in flux-limited neutrino diffusion and two-moment transport,
\newblock \emph{Astronomy and Astrophysics}, \textbf{356} (2000), 559--569.

\end{thebibliography}

\medskip
Received xxxx 20xx; revised xxxx 20xx.
\medskip

\end{document}